\documentclass[a4paper,reqno]{amsart}

\usepackage{amsthm, amsfonts, amsmath, amssymb}
\usepackage{amsfonts}
\usepackage{mathtools}
\usepackage[english]{babel}
\usepackage{epsfig}
\usepackage{xcolor}
\usepackage[headings]{fullpage}
\usepackage{multirow}
\usepackage{hyperref}

\DeclareMathOperator{\ad}{ad}
\DeclareMathOperator{\Aut}{Aut}
\DeclareMathOperator{\Der}{Der}
\DeclareMathOperator{\diag}{diag}
\DeclareMathOperator{\Dih}{Dih}
\DeclareMathOperator{\GL}{GL}
\DeclareMathOperator{\I}{I}
\DeclareMathOperator{\id}{id}
\DeclareMathOperator{\im}{im}
\DeclareMathOperator{\OO}{O}
\DeclareMathOperator{\rank}{rank}
\DeclareMathOperator{\Ric}{Ric}

\DeclareMathOperator{\Sym}{Sym}
\DeclareMathOperator{\T}{T}

\newcommand{\bb}{\mathbb}
\newcommand{\mf}{\mathfrak}

\theoremstyle{plain}
\newtheorem{theorem}{Theorem}[section]
\newtheorem{proposition}[theorem]{Proposition}

\newtheorem{corollary}[theorem]{Corollary}

\theoremstyle{definition}

\newtheorem{example}[theorem]{Example}

\theoremstyle{remark}
\newtheorem{remark}[theorem]{Remark}

\numberwithin{equation}{section}

\begin{document}

\title[The moduli space of left-invariant metrics on CSLAs]{The moduli
  space of left-invariant metrics on six-dimensional characteristically
  solvable nilmanifolds}

\author{Isolda Cardoso}
\address{Universidad Nacional de Rosario, ECEN-FCEIA,
  Departamento de Ma\-te\-má\-ti\-ca. Av. Pellegrini 250, 2000
  Rosario, Argentina.}
\email{\href{mailto:isolda@fceia.unr.edu.ar}{isolda@fceia.unr.edu.ar}}

\author{Ana Cosgaya}
\address{CONICET and Universidad Nacional de Rosario, EFB-FCEIA,
  Departamento de Ma\-te\-má\-ti\-ca. Av. Pellegrini 250, 2000
  Rosario, Argentina.}
\email{\href{mailto:acosgaya@fceia.unr.edu.ar}{acosgaya@fceia.unr.edu.ar}}

\author{Silvio Reggiani}
\address{CONICET and Universidad Nacional de Rosario, ECEN-FCEIA,
  Departamento de Ma\-te\-má\-ti\-ca. Av. Pellegrini 250, 2000
  Rosario, Argentina.}
\email{\href{mailto:reggiani@fceia.unr.edu.ar}{reggiani@fceia.unr.edu.ar}}
\urladdr{\url{http://www.fceia.unr.edu.ar/~reggiani}}

\date{\today}

\thanks{Supported by CONICET and UNR. Partially supported by SeCyT-UNR and ANPCyT}



\begin{abstract}
  A real Lie algebra is said to be characteristically solvable if its derivation algebra is solvable. We explicitly determine the moduli space of left-invariant metrics, up to isometric automorphism, for $6$-dimensional nilmanifolds whose associated Lie algebra is characteristically solvable of triangular type. We also compute the corresponding full isometry groups. For each left-invariant metric on these nilmanifolds we compute the index and distribution of symmetry. In particular, we find the first known examples of Lie groups which do not admit a left-invariant metric with positive index of symmetry. As an application we study the index of symmetry of nilsoliton metrics. We prove that nilsoliton metrics detect the existence of left-invariant metrics with positive index of symmetry.
\end{abstract}

\maketitle

\section{Introduction}\label{Sec:Introduction}  

Let $H$ be a Lie group with Lie algebra $\mf h$. A very natural problem to study is: how many essentially different left-invariant geometries does $H$ admit? More precisely, one wants to determine the moduli space $\mathcal M(H) / {\sim}$ of left-invariant metrics on $H$ up to isometric isomorphism and understand its topological structure. Solving this problem is extremely hard and the answer is unknown in general, even for Lie groups of low dimension. Some partial results are present in the literature. For instance in \cite{lauretDegenerationsLieAlgebras2003} and \cite{kodamaSpaceLeftinvariantMetrics2011}, the Lie groups $H$ with low dimensional $\mathcal M(H) / {\sim}$ are classified. In \cite{haLeftInvariantMetrics2009}, $\mathcal{M}(H) / {\sim}$ is completely determined when $\dim H = 3$. In the nilpotent case we can mention the work \cite{discalaInvariantMetricsIwasawa2013} where the problem was solved for the Iwasawa manifold, and more recently in \cite{reggianiModuliSpaceLeftinvariant2020} for some $6$-dimensional nilpotent Lie groups with first Betti number equal to $4$ and \cite{figulaIsometryClassesSimply2018,ficzereIsometryGroupsSixdimensional2023} for a certain nilpotent $6$-dimensional Lie algebras generalizing the filiform family. In \cite{consoleModuliSpaceSixDimensional2005} a description of the moduli space of nilpotent metric Lie algebras in dimension $\le 6$ is given. This last result is closely related with the problem mentioned above, but the approach is different since we look for a precise description of $\mathcal M(H)/{\sim}$ for a fixed isomorphism class. 

When a description of the moduli space of left-invariant metrics on $H$ is available, we can study how different (invariant) geometric objects vary along different left-invariant geometries. Some examples of such objects are Hermitian structures (in even dimension), the signature of the Ricci operator or the so-called index of symmetry. This last object is of particular interest to us. Let us briefly say that the index of symmetry $i_{\mf s}(M)$ of a homogeneous Riemannian manifold $M$ is a geometric invariant that measures how far $M$ is from being a symmetric space. There is a strong structural theory regarding the index (or the co-index $\dim M - i_{\mf s}(M)$) of symmetry for compact homogeneous spaces, and there is also a classification of compact homogeneous spaces with co-index of symmetry $3$. We refer to \cite{olmosIndexSymmetryCompact2014,berndtCompactHomogeneousRiemannian2017,reggianiManifoldsAdmittingMetric2021} for more details on this topic. Despite some progress made for left-invariant metrics on $3$-dimensional solvable Lie groups and naturally reductive nilpotent Lie groups (see \cite{reggianiDistributionSymmetryNaturally2019,mayIndexSymmetryLeftinvariant2021,cosgayaIsometryGroupsThreedimensional2022}), a general theory for the index of symmetry in the non-compact case is virtually unknown. Moreover, even the classificatory results for low co-index in the non-compact case are sparse. 

In this paper we deal with characteristically solvable Lie algebras of triangular type, CSLAT for short, namely nilpotent Lie algebras whose derivation algebra is isomorphic to a subalgebra of the triangular Lie algebra. In dimension~$6$, the family of CSLATs has 20 of the 34 isomorphism classes of (real) nilpotent Lie algebras. Our main goal is to classify the left-invariant metrics up to isometric automorphism in every CSLAT of dimension $6$ as well as computing the index of symmetry of every metric in this classification. 

In order to determine the moduli space $\mathcal M(H) / {\sim}$ when $H$ is $6$-dimensional with $\mf h$ CSLAT, we compute in Section \ref{sec:full-aut-group} the full automorphism group $\Aut(\mf h)$. This includes an explicit description of $D = \Aut(\mf h) / {\Aut_0(\mf h)}$ which can be realized as a finite subgroup of $\Aut(\mf h)$ intersecting every connected component exactly once. Moreover, it follows from Theorem~\ref{sec:full-autom-groups-1} that $D$ is always a $2$-group and $\Aut(\mf h)$ is isomorphic to a subgroup of the lower triangular group $\T_6$ if and only if $D$ is abelian. We must notice that there is some previous work dealing with the computation of the automorphism group of nilpotent Lie algebras, see for instance \cite{magninComplexStructuresIndecomposable2007}. However, as far as we know,  in the existing literature there is not an explicit computation of the automorphism group for all of the CSLAT. We present here an algorithmic  procedure to compute $\Aut(\mf h)$ with a key simplification process. From this we obtain a nice presentation of the automorphism group with respect to a suitable basis of $\mf h$, which plays an important role in determining the moduli space of left invariant-metrics. 

In Section \ref{sec:moduli-space-left-1} we address the problem of determining the moduli space of left-invariant metrics on $H$. We start by describing the moduli space $\mathcal M(H)/ {\sim_0}$ of left-invariant metrics on $H$ modulo an isometric automorphism in the connected component of $\Aut(\mf h)$. We prove that $\mathcal M(H) / {\sim_0}$ is a smooth manifold diffeomorphic to a certain embedded submanifold $\Sigma$ of $\T_6^+$ (the subgroup of $\T_6$ with positive entries on the diagonal). Then we explain how to obtain $\mathcal M(H) / {\sim}$ from $\mathcal M(H) / {\sim_0}$ as a finite quotient. It turns out that in most cases, $\mathcal M(H) / {\sim}$ is homeomorphic to the quotient of $\Sigma$ modulo a finite group acting as reflections on $\T_6^+$, which leaves $\Sigma$ invariant.   

In Section \ref{sec:full-isometry-group-1} we compute the full isometry group $\I(H, g_\sigma)$, where $g_\sigma$ is a left-invariant metric representing $\sigma \in \Sigma$. It is known from \cite{wolfLocallySymmetricSpaces1963} that $\I(H, g_\sigma) \simeq H \rtimes K$, where $K = \Aut(\mf h) \cap \OO(g_\sigma)$. We prove that if $D$ is abelian, then $K \subset D$. In particular, when $D$ is abelian there is a finite number of subgroups of $\Aut(\mf h)$ that can serve as the full isotropy subgroup of $g_\sigma$, for all $\sigma \in \Sigma$. This does not hold if $D$ is not abelian. Indeed, we provide examples of $1$-parameter families of subgroups $K_r \subset \Aut(\mf h)$ and $\sigma_r \in \Sigma$ such that $K_r$ is the full isotropy group of $g_{\sigma_r}$. The explicit determination of the full isotropy groups $K$ for any left-invariant metric is also possible and it is treated in that section.   

In Section \ref{sec:index-symmetry-csla-title} we compute the index of symmetry of left-invariant metrics in the context of CSLATs of dimension $6$. We remark that for every Lie algebra isomorphism class, except for five distinguished cases, every left-invariant metric on $H$ has trivial index of symmetry. This provides the first known examples of Lie groups which do not admit a left-invariant metric with positive index of symmetry. On the other side, among the CSLATs whose associated Lie groups do admit metrics with non-trivial index of symmetry we can find new examples of non-compact homogeneous spaces with co-index of symmetry $3$ and $4$. Lastly, we find some examples of left-invariant metrics on nilpotent Lie groups whose distribution of symmetry (i.e.\ the left-invariant distribution induced by the Killing field parallel at the identity element) is not contained in the center of the Lie algebra. It was not clear that nilpotent Lie groups with that property actually exist.

As an application of our results, in Section \ref{sec:application-to-nilsolitons} we study the index and distribution of symmetry of nilsoliton metrics in CSLATs. Recall that a left-invariant metric on $H$ is called a \emph{nilsoliton} metric if its Ricci operator satisfies $\Ric = c \id_{\mf h} + D$ for some $c \in \nobreak \bb R$ and $D \in \Der(\mf h)$. These metrics are particular cases of \emph{solvsoliton} metrics (same definition but for solvable Lie groups), a family that exhausts the homogeneous expanding Ricci soliton metrics (see \cite{bohmNoncompactEinsteinManifolds2023}). Solvsoliton metrics are proved to be unique (if they exist) up to scaling. Moreover, nilsoliton metrics are given by the nilradicals of Einstein solvmanifolds (see \cite{lauretRicciSolitonSolvmanifolds2011}). A classification of nilsoliton metrics in dimension $6$ can be found in \cite{willRankoneEinsteinSolvmanifolds2003}. Since nilsoliton metrics are Ricci soliton metrics, they are nice enough so that they are not improved under the Ricci flow (that is, evolving only by scaling or pulling-back by a diffeomorphism). Thus it is expected that nilsoliton metrics have low co-index of symmetry whenever the underlying Lie algebra structure do not obstruct the existence of positive index of symmetry. We prove that this is indeed the case. More precisely, if $H$ admits metrics with positive index of symmetry, then the nilsoliton metric has positive index of symmetry. It is interesting to notice that the index of symmetry of nilsoliton metrics is not always maximal among all the left-invariant metrics. However, the distribution of symmetry of nilsoliton metrics is well behaved, in the sense that it is contained in the center of the Lie algebra when this is not obstructed by the isomorphism class of $\mf h$.

Finally, in Appendix \ref{sec:appendix} we include some tables summarizing the classificatory results obtained in the paper. This results often rely in heavy computations that we performed with the computer software SageMath. The source code of such computation is freely available and provided in the separate GitHub repository \cite{cardosoAuxiliarySageMathNotebooks2024}. We believe that some of this code could be of use in studying other geometric problems in nilpotent Lie groups of low dimension. 

\section{Characteristically solvable Lie algebras}

\par In this article we mainly deal with $6$-dimensional real
nilpotent Lie algebras. This is the greatest dimension with a finite
number of isomorphism classes: there are $34$ isomorphism classes of
nilpotent real Lie algebras. Observe that every $6$-dimensional
nilpotent Lie algebra admits a basis $e_1, \ldots, e_6$ for which the
structure coefficients are $-1$, $0$ or $1$. So, it is convenient to
denote a Lie algebra with a tuple summarizing its structure. We will
follow the notation used in
\cite{salamonComplexStructuresNilpotent2001}. For example, the
notation $\mf h_{11} = (0, 0, 0, 12, 13, 14 + 23)$ means that
$\mf h_{11}$ is the nilpotent Lie algebra $(\bb R^6, [\cdot, \cdot])$
where the Lie bracket in the canonical basis is given by the
non-trivial relations
\begin{align*}
  [e_1, e_2] = -e_3 && [e_1, e_3] = -e_4 && [e_1, e_4] = [e_2, e_3] = -e_6.
\end{align*}

\par In Table \ref{tab:classif-nilp-dim-6} we list all the $34$ isomorphism classes of $6$-dimensional nilpotent Lie algebras. Note that some of the structure coefficients are indicated with an asterisk. This means that further on the paper we shall consider a different basis (given in the last column) to better serve our purposes. We also mention that the third column on the table shows the nilpotency step and the fourth column indicates the dimension of the commuting ideals for the decomposable Lie algebras.

\begin{table}[ht] 
  \caption{Nilpotent and CSLAT Lie algebras of dimension $6$}
  \label{tab:classif-nilp-dim-6}
\centering
    \(
  \begin{array}{|c|l|c|c|c|l|}
    \hline
    \text{Name}  & \text{Structure coefficients} & \text{Step} & \oplus & \text{CSLAT} & \text{Standard basis} \\ \hline \hline
    \mf h_1      & (0,0,0,0,0,0)                 & 1           & 1 + \cdots + 1 & & \\ \hline
    \mf h_2      & (0,0,0,0,12,34)               & 2           & 3 + 3 & & \\ \hline
    \mf h_3      & (0,0,0,0,0,12+34)             & 2           & 1 + 5 & & \\ \hline
    \mf h_4      & (0,0,0,0,12,14+23)            & 2           & & & \\ \hline
    \mf h_5      & (0,0,0,0,13+42,14+23)         & 2           & & & \\ \hline
    \mf h_6      & (0,0,0,0,12,13)               & 2           & 1 + 5 & & \\ \hline
    \mf h_7      & (0,0,0,12,13,23)              & 2           & & & \\ \hline
    \mf h_8      & (0,0,0,0,0,12)                & 2           & 1 + 1 + 1 + 3 & & \\ \hline
    \mf h_9      & (0,0,0,0,12,14+25)^*          & 3           & 1 + 5 & \checkmark & (0,0,0,0,12,51+23)\\ \hline
    \mf h_{10}   & (0,0,0,12,13,14)              & 3           & & \checkmark & (0,0,0,12,13,14) \\ \hline
    \mf h_{11}   & (0,0,0,12,13,14+23)           & 3           & & \checkmark & (0,0,0,12,13,14+23) \\ \hline
    \mf h_{12}   & (0,0,0,12,13,24)              & 3           & & \checkmark & (0,0,0,12,13,24) \\ \hline
    \mf h_{13}   & (0,0,0,12,13+14,24)           & 3           & & \checkmark & (0,0,0,12,13+14,24) \\ \hline
    \mf h_{14}   & (0,0,0,12,14,13+42)           & 3           & & \checkmark & (0,0,0,12,14,13+42) \\ \hline
    \mf h_{15}   & (0,0,0,12,13+42,14+23)         & 3           & & & \\ \hline
    \mf h_{16}   & (0,0,0,12,14,24)              & 3           & 1 + 5 & & \\ \hline
    \mf h_{17}   & (0,0,0,0,12,15)               & 3           & 1 + 1 + 4 & & \\ \hline
    \mf h_{18}   & (0,0,0,12,13,14+35)^*         & 3           & & \checkmark & (0,0,0,12,13,15+24) \\ \hline    
    \parbox[m][1.2pc][c]{1.1pc}{$\mf h_{19}^-$} & (0,0,0,12,23,14-35)           & 3           & & & \\ \hline  
    \parbox[m][1.2pc][c]{1.1pc}{$\mf h_{19}^+$} & (0,0,0,12,23,14+35)^*         & 3           & & \checkmark & (0,0,0,23,21,14+35) \\ \hline  
    \mf h_{20}   & (0,0,0,0,12,15+34)            & 3           & & & \\ \hline     
    \mf h_{21}   & (0,0,0,12,14,15)              & 4           & 1 + 5 & \checkmark & (0,0,0,12,14,15) \\ \hline       
    \mf h_{22}   & (0,0,0,12,14,15+24)           & 4           & 1 + 5 & \checkmark & (0,0,0,12,14,15+24) \\ \hline    
    \mf h_{23}   & (0,0,12,13,23,14)             & 4           & & \checkmark & (0,0,12,13,23,14) \\ \hline      
    \mf h_{24}   & (0,0,0,12,14,15+23+24)        & 4           & & \checkmark & (0,0,0,12,14,15+23+24) \\ \hline 
    \mf h_{25}   & (0,0,0,12,14,15+23)           & 4           & & \checkmark & (0,0,0,12,14,15+23) \\ \hline   
    \parbox[m][1.2pc][c]{1.1pc}{$\mf h_{26}^-$} & (0,0,12,13,23,14-25)^*        & 4           & & \checkmark & (0,0,12,31,32,15+24) \\ \hline
    \parbox[m][1.2pc][c]{1.1pc}{$\mf h_{26}^+$} & (0,0,12,13,23,14+25)          & 4           & & & \\ \hline
    \mf h_{27}   & (0,0,0,12,14-23,15+34)        & 4           & & \checkmark & (0,0,0,12,14-23,15+34) \\ \hline
    \mf h_{28}   & (0,0,12,13,14,15)             & 5           & & \checkmark & (0,0,12,13,14,15) \\ \hline     
    \mf h_{29}   & (0,0,12,13,14,23+15)          & 5           & & \checkmark & (0,0,12,13,14,23+15) \\ \hline  
    \mf h_{30}   & (0,0,12,13,14+23,24+15)       & 5           & & \checkmark & (0,0,12,13,14+23,24+15) \\ \hline
    \mf h_{31}   & (0,0,12,13,14,34+52)          & 5           & & \checkmark & (0,0,12,13,14,34+52) \\ \hline  
    \mf h_{32}   & (0,0,12,13,14+23,34+52)       & 5           & & \checkmark & (0,0,12,13,14+23,34+52) \\ \hline                                                          
  \end{array}
  \)
\end{table}

\par Let $\mf h$ be a nilpotent Lie algebra. We say that $\mf h$ is
\emph{characteristically solvable, CSLA} for short, if its derivation
algebra $\Der(\mf h)$ is a solvable Lie algebra. Notice that our definition is slightly different from the definition given in \cite{togoDerivationAlgebrasLie1961}. In addition, we say that a CSLA $\mf h$ is of \emph{triangular type}, denoted by \emph{CSLAT}, if $\operatorname{spec}(D) \subset \mathbb R$ for every $D \in \Der(\mf h)$. The next result
gives the classification of CSLAs in dimension $6$. Observe that most CSLAs turn out to be of triangular type.

\begin{proposition}\label{sec:char-solv-lie}
  \it Let $\mf h$ be a real nilpotent Lie algebra of dimension
  $6$. Then $\mf h$ is a CSLA if and only if $\mf h$ is isomorphic to
  one of the following Lie algebras:
  \begin{align*}
    \mathfrak{h}_9, \mathfrak{h}_{10}, \mathfrak{h}_{11}, \mathfrak{h}_{12}, \mathfrak{h}_{13}, \mathfrak{h}_{14}, \mathfrak h_{15}, \mathfrak{h}_{18}, \mathfrak{h}_{19}^-, \mathfrak{h}_{19}^+, \mathfrak{h}_{21}, \mathfrak{h}_{22}, \mathfrak{h}_{23}, \mathfrak{h}_{24}, \mathfrak{h}_{25}, \mathfrak{h}_{26}^-, \mathfrak{h}_{26}^+, \mathfrak{h}_{27}, \mathfrak{h}_{28}, \mathfrak{h}_{29}, \mathfrak{h}_{30}, \mathfrak{h}_{31}, \mathfrak{h}_{32}. 
  \end{align*} 
  Moreover, each of the above CSLAs is a CSLAT except for $\mathfrak h_{15}$, $\mathfrak h_{19}^-$ and $\mathfrak h_{26}^+$.
\end{proposition}

\begin{proof}
  \par By a direct computation, we can see that the derivation Lie
  algebras for $\mf h_{10}, \ldots, \mf h_{14}$,
  $\mf h_{21}, \ldots, \mf h_{25}$, $\mf h_{27}, \ldots, \mf h_{32}$
  are represented by triangular matrices in the standard basis.
  Consider the following changes of basis,
  \begin{align*}
    e_1' & = e_2, & e_2' & = e_1, & e_3' & = e_4, & e_4' & = e_3, & e_5' & = e_5, & e_6' & = e_6 \\
    e_1' & = e_1, & e_2' & = e_3, & e_3' & = e_2, & e_4' & = e_5, & e_5' & = e_4, & e_6' & = e_6 
  \end{align*}
  for $\mf h_9$ and $\mf h_{18}$, respectively,
  \begin{align*}
    e_1' & = e_1 + e_3, & e_2' & = e_2,       & e_3' & = e_1 - e_3, &
    e_4' & = e_4 + e_5, & e_5' & = e_4 - e_5, & e_6' & = 2e_6 
  \end{align*}
  for $\mf h_{19}^+$ and
  \begin{align*}
    e_1' & = e_1 + e_2,    & e_2' & = e_1 - e_2,    & e_3' & = -2e_3, &
    e_4' & = 2(e_4 + e_5), & e_5' & = 2(e_4 - e_5), & e_6' & = 4e_6 
  \end{align*}
  for $\mf h_{26}^-$. These changes of basis simultaneously
  triangularize all the derivations. The Lie algebras $\mf h_{15}$, $\mf h_{19}^-$ and $\mf h_{26}^+$ are characteristically solvable since the commutator of its derivation algebra is nilpotent. However, these algebras admit non trivial skew-symmetric derivations (with respect to the inner product induced
  by the standard basis). Thus, these Lie algebras cannot be of triangular type. Finally, in order to see that the remaining nilpotent Lie algebras of dimension $6$ are not characteristically solvable, one can check that its derivation algebra has a Lie subalgebra isomorphic to $\mf{sl}_2(\mathbb{R})$. The detailed computations are long and
  tedious to be made by hand and were instead performed with the
  software SageMath. The corresponding Jupyter notebook can be found
  in \cite[Notebook 01]{cardosoAuxiliarySageMathNotebooks2024}, the GitHub repository supporting this article.
\end{proof}

\par In Table \ref{tab:classif-nilp-dim-6} we also list all the CSLATs. From now on we will consider as the standard basis the one with the structure coefficients given in the last column of this table (i.e., the standard basis is the same as in the second column except for $\mf h_9, \mf h_{18}, \mf h_{19}^+$ and $\mf h_{26}^-$ which is changed according to the proof of Proposition \ref{sec:char-solv-lie}).

\begin{corollary}
  \it Let $\mf h$ be an $s$-step nilpotent Lie algebra of dimension
  $6$. If $\mf h$ is a CSLA, then $s \ge 3$.
\end{corollary}

\section{The full automorphism group of a CSLAT}
\label{sec:full-aut-group}

\par Let us fix a CSLAT $\mf h$ of dimension $6$. We denote by
$\T_6 \subset \GL_6(\bb R)$ the Lie subgroup of lower triangular
matrices and by $\mf t_6$ its Lie algebra. In order to compute the
full automorphism group $\Aut(\mf h)$ of $\mf h$ we proceed as
follows. According to the proof of Proposition
\ref{sec:char-solv-lie}, up to a suitable change of basis, we can
identify $\Der(\mf h)$ with a Lie subalgebra of $\mf t_6$. Then
$e^{\Der(\mf h)} = \Aut_0(\mf h)$ is the connected component of the
automorphism group. Moreover, if we denote by $\Der_{\mf d}(\mf h)$
the abelian subalgebra of (simultaneously) diagonalizable derivations,
then we see that $\Aut(\mf h)$ has at least $2^k$ connected
components, where $k = \dim \Der_{\mf d}(\mf h)$. We will prove in
Theorem \ref{sec:full-autom-groups-1} that this bound is almost always
attained.

\begin{remark}\label{sec:full-autom-groups-2}
  \par Let $\mf h$ be a CSLAT of dimension $6$. Since $\Aut(\mf h)$ is an algebraic group, it has finitely many connected components, hence there exists a finite subgroup of automorphisms $D$ such that $\Aut(\mf h) = \Aut_0(\mf h) \rtimes D$. Indeed, it follows from \cite[Theorem~3.1]{hochschildStructureLieGroups1965} that a maximal compact subgroup of $\Aut(\mf h)$ meets every connected component. 
  
  Recall that $D$ must be conjugated to a subgroup of $\OO(n)$, however this does not imply that $D$ consists of diagonal automorphisms in the standard basis.
\end{remark}

\begin{theorem}\label{sec:full-autom-groups-1}
  \it Let $\mf h$ be a CSLAT of dimension $6$ such that
  $\mf h \not\simeq \mf h_{13}$, $\mf h \not\simeq \mf h_{19}^+$ and
  $\mf h \not\simeq \mf h_{26}^-$. Then $\Aut(\mf h)$ is isomorphic to
  a subgroup of $\T_6$. Moreover,
  \begin{enumerate} 
  \item $\Aut(\mf h) / {\Aut_0(\mf h)} \simeq (\bb Z_2)^k$ where
    $k = \dim \Der_{\mf d}(\mf h)$.
  \item
    $\Aut(\mf h_{13}) / {\Aut_0(\mf h_{13})} \simeq \mathrm G_{16}^3 =
    (\mathbb Z_2 \oplus \mathbb Z_2) \rtimes \mathbb Z_4$.
  \item
    $\Aut(\mf h_{19}^+) / {\Aut_0(\mf h_{19}^+)} \simeq {\Dih_4}
    \times \bb Z_2$.
  \item $\Aut(\mf h_{26}^-) / {\Aut_0(\mf h_{26}^-)} \simeq \Dih_4$.
  \end{enumerate}
\end{theorem}

\begin{proof}
  \par Let $\mf h$ be a CSLAT of dimension $6$. It follows from Remark
  \ref{sec:full-autom-groups-2} that there exists a finite subgroup
  $D$ of $\Aut(\mf h)$ which meets every connected component exactly
  once. So $\Aut(\mf h) / {\Aut_0(\mf h)} \simeq D$ and as we observed
  before, $D$ has a subgroup isomorphic to $(\bb Z_2)^k$, where
  $k = \dim \Der_{\mf d}(\mf h)$. We will prove $D = (\bb Z_2)^k$ for
  all $\mf h \neq \mf h_{13}, \mf h_{19}^+, \mf h_{26}^-$. Let $e_1, \ldots, e_6$
  be the standard basis of $\mf h$ (cfr.\
  Table~\ref{tab:classif-nilp-dim-6}). Let $s$ be the least natural number
  such that the Lie subalgebra generated by $e_1, \ldots, e_s$
  coincides with $\mf h$. One can easily see that $s = 2, 3$ or
  $4$. Let $\varphi \in \Aut(\mf h)$, then $\varphi$ is lower
  triangular in the standard basis if and only if
  $e^i(\varphi(e_j)) = 0$ for all $1 \le i < j \le s$, where
  $e^1, \ldots, e^6$ is the dual basis of the standard basis. We argue
  by cases according to the possible values of $s$. Let
  $X = \sum \alpha_i e_i$ be an arbitrary element of $\mf h$.

  \par \textbf{Case 1: $s = 2$.} In this case,
  $\mf h \in \{\mf h_{23}, \mf h_{26}^-, \mf h_{28}, \mf h_{29}, \mf
  h_{30}, \mf h_{31}, \mf h_{32}\}$. Then, for all $\mf h$ except
  $\mf h_{26}^-$ and $\mf h_{32}$ we have that
  $\rank(\ad_{e_1}) > \rank(\ad_{e_2})$ and if $\alpha_1 \neq 0$ then
  $\rank( \ad_X) \ge \rank(\ad_{e_1})$. So, in these cases, we have
  $e^1(\varphi(e_2)) = 0$ and $\varphi \in \T_6$. Moreover,
  $D \simeq (\bb Z_2)^k$ where $k = 1$ for
  $\mf h = \mf h_{29}, \mf h_{30}$ and $k = 2$ for
  $\mf h = \mf h_{23}, \mf h_{28}, \mf h_{31}$. In fact, it is easy to
  see that if $k = 1$, then $\varphi(e_2)$ is determined by
  $\varphi(e_1)$, hence $D$ is isomorphic to the subgroup generated by
  the automorphism $\varphi_1: e_1 \mapsto -e_1$. If $k = 2$ one can
  see that $D$ is isomorphic to the subgroup generated by the
  automorphisms $\varphi_1, \varphi_2$ given by
  \begin{align}\label{eq:1}
    \varphi_1(e_1) = -e_1 = -\varphi_2(e_1) && \text{and} && \varphi_1(e_2) = e_2 = -\varphi_2(e_2).
  \end{align}
  Let us look at the remaining cases.

  \par \textbf{Case 1.1: $\mf h = \mf h_{26}^-$.} In this case
  $\rank(\ad_{e_1}) = \rank(\ad_{e_2}) = 3$. Moreover, if
  $\rank(\ad_X) = 3$, then $\alpha_1 \neq 0$ or $\alpha_2 \neq
  0$. Since $(\ad_{e_1})^3 = (\ad_{e_2})^3 = 0$ and
  $(\ad_X)^3 = 0$ if and only if $\alpha_1 \alpha_2 = 0$, we have that
  \begin{align*}
    e^1(\varphi(e_1)) \neq 0 & \implies e^2(\varphi(e_1)) = e^1(\varphi(e_2)) = 0 \neq e^2(\varphi(e_2)) \\
    e^2(\varphi(e_1)) \neq 0 & \implies e^1(\varphi(e_1)) = e^2(\varphi(e_2)) = 0 \neq e^1(\varphi(e_2)).
  \end{align*}
  So in addition to the automorphisms $\varphi_1, \varphi_2 \in D$
  defined by the same formula as in \eqref{eq:1} we can define an
  automorphism $\varphi_3 \in D$ by
  \begin{align*}
    \varphi_3(e_1) = e_2 && \varphi_3(e_2) = e_1. 
  \end{align*}
  The above computations show that $D$ is generated by
  $\varphi_1, \varphi_2, \varphi_3$. Moreover, note that $D$ is also
  generated by $\varphi_1$ and $\varphi_1 \circ \varphi_3$, and since
  $\varphi_1^2 = (\varphi_1 \circ \varphi_3)^4$ we conclude that
  $D \simeq \Dih_4$.

  \par \textbf{Case 1.2: $\mf h = \mf h_{32}$.} Notice that
  $Z(\mf h) = \bb R e_6$ and $\im(\ad_{e_2})^3 = Z(\mf h)$. This
  property is invariant under automorphisms and
  $\im (\ad_X)^3 = Z(\mf h)$ implies $\alpha_1 = 0$. So
  $e^1(\varphi(e_2)) = 0$ and reasoning as in some previous cases,
  $D \simeq \bb Z_2$.

  \par \textbf{Case 2: $s = 3$.} In this case
  \begin{equation*}
    \mf h \in \{\mf h_{10}, \mf h_{11}, \mf h_{12}, \mf h_{13}, \mf h_{14},
    \mf h_{18}, \mf h_{19}^+, \mf h_{21}, \mf h_{22}, \mf h_{24}, \mf h_{25}, 
    \mf h_{27}\}.
  \end{equation*}
  Routine computations show that
  $\rank(\ad_{e_1}) > \rank(\ad_{e_i})$, $i = 2, 3$, for all $\mf h$
  except $\mf h_{12}$, $\mf h_{13}$ and $\mf h_{19}^+$. Also, we have
  that $\rank(\ad_{e_2}) > \rank(\ad_{e_3})$ for all $\mf h$ except
  $\mf h_{10}$, $\mf h_{11}$, $\mf h_{19}^+$ and $\mf h_{27}$.

  \par \textbf{Case 2.1:
    $\rank(\ad_{e_1}) > \rank(\ad_{e_2}) > \rank(\ad_{e_3})$.} This
  case is analogous to the generic case with $s = 2$, since
  $\alpha_1 \neq 0$ implies $\rank(\ad_X) \ge \rank(\ad_{e_1})$ and
  $\alpha_2 \neq 0$ implies $\rank(\ad_X) \ge \rank(\ad_{e_2})$. This
  shows that every $\varphi$ is triangular with respect to the
  standard basis, and $D \simeq (\bb Z_2)^3$ is generated by
  $\varphi_1, \varphi_2, \varphi_3$ defined by
  \begin{align}\label{eq:2}
    \varphi_1(e_1) & = -e_1 && \varphi_1(e_2) = e_2  && \varphi_1(e_3) = e_3 \notag \\
    \varphi_2(e_1) & = e_1  && \varphi_2(e_2) = -e_2 && \varphi_2(e_3) = e_3 \\
    \varphi_3(e_1) & = e_1  && \varphi_3(e_2) = e_2  && \varphi_3(e_3) = -e_3 \notag  
  \end{align}
  
  \par \textbf{Case 2.2:
    $\rank(\ad_{e_1}) > \rank(\ad_{e_2}) = \rank(\ad_{e_3})$.} Here we
  have three subcases.

  \par \textbf{Case 2.2.1: $\mf h = \mf h_{10}$.}  In this case we have $\rank(\ad_{e_1}) = 3$ and $\rank(\ad_{e_2}) = \rank(\ad_{e_3}) =\nobreak  1$. Moreover, if $\alpha_1 \neq 0$ then $\rank(\ad_X) = 3$, and hence $e^1(\varphi(e_2)) = e^1(\varphi(e_3)) = 0$. Also notice that $\im(\ad_{e_3}) \subset Z(\mf h)$, but $\alpha_2 \neq 0$ implies $\im(\ad_X) \not\subset Z(\mf h)$, so $e^2(\varphi(e_3)) = 0$ and $\varphi$ is triangular. One can check that $D$ is generated by $\varphi_1, \varphi_2, \varphi_3$ defined with the same formulas as in \eqref{eq:2}.

  \par \textbf{Case 2.2.2: $\mf h = \mf h_{11}$.} Here the same
  argument as in the above case works for proving that
  $\varphi \in T_6$. However, in this case $D \simeq (\bb Z_2)^2$
  since $[e_2, e_3] = -[e_1, [e_1, e_2]]$ implies
  $e^3(\varphi(e_3)) = e^1(\varphi(e_1))^2$.

  \par \textbf{Case 2.2.3: $\mf h = \mf h_{27}$.} In this case we can
  use the following variation of the above argument. Let $C^2(\mf h)$
  be the third ideal in the lower central series of $\mf h$. Then
  $\im(\ad_{e_3}) \subset C^2(\mf h)$ and $\alpha_2 \neq 0$ implies
  $\im(\ad_X) \not\subset C^2(\mf h)$. As in the previous case, we can
  see that $D \simeq (\bb Z_2)^2$.
    
  \par \textbf{Case 2.3:
    $\rank(\ad_{e_1}) = \rank(\ad_{e_2}) > \rank(\ad_{e_3})$.} We have
  to consider the following sub-cases.

  \par \textbf{Case 2.3.1: $\mf h = \mf h_{12}$.} Recall that
  $(\ad_{e_1})^2 = 0 \neq (\ad_{e_2})^2$ and $(\ad_X)^2 = 0$ if and
  only if $\alpha_2 = 0$. This shows $e^1(\varphi(e_2)) = 0$. In order
  to see that $e^1(\varphi(e_3)) = e^2(\varphi(e_3)) = 0$ notice that
  $\alpha_1 \neq 0$ or $\alpha_2 \neq 0$ imply $\rank(\ad_X) \ge 2$,
  but $\rank(\ad_{e_3}) = 1$. In this case $D = (\bb Z_2)^3$ with
  generators $\varphi_1, \varphi_2, \varphi_3$ defined as in~\eqref{eq:2}.

  \par \textbf{Case 2.3.2: $\mf h = \mf h_{13}$.} This case is more
  complicated than the previous ones since it cannot be solved using
  linear invariants. First notice that $\alpha_1 \neq 0$ or
  $\alpha_2 \neq 0$ implies
  $\operatorname{rank}(\operatorname{ad}_X) \ge 2$. This shows
  $e^1(\varphi(e_3)) = e^2(\varphi(e_3)) = 0$. Also since $\varphi$
  preserves the commutator and the center, one has
  $e^i(\varphi(e_4)) = 0$ for $i = 1, 2, 3$ and
  $e^i(\varphi(e_j)) = 0$ for $i = 1, 2, 3, 4$ and $j = 5, 6$. Now we
  use a brute force approach to recover the general form of the
  automorphism $\varphi$. When $e^1(\varphi(e_2)) = 0$ we obtain a
  subgroup of $D$ isomorphic to $\mathbb Z_2 \oplus \mathbb Z_2$ which
  is generated by the automorphisms $\varphi_1, \varphi_2$ given by
  \begin{align*}
    \varphi_1(e_1) & = -e_1, && \varphi_1(e_2) = e_2,  && \varphi_1(e_3) = -e_3, \\
    \varphi_2(e_1) & = e_1,  && \varphi_2(e_2) = -e_2, && \varphi_2(e_3) = -e_3. 
  \end{align*}

  Finally, when $e^1(\varphi(e_2)) \neq 0$ we obtain a subgroup of $D$
  isomorphic to $\mathbb Z_4$ with a generator $\varphi_3$ given by
  \begin{align*}
    \varphi_3(e_1) & = e_2, && \varphi_3(e_2) = -e_1,  && \varphi_3(e_3) = -e_3 + e_4. 
  \end{align*}

  \par Notice that
  ${\varphi_3} \circ \varphi_1 \circ {\varphi_3^{-1}} = \varphi_2$. We
  refer to the corresponding SageMath notebook for the details.
  
  \par \textbf{Case 2.4:
    $\rank(\ad_{e_1}) = \rank(\ad_{e_2}) = \rank(\ad_{e_3})$.} The
  only CSLAT in this case is $\mf h = \mf h_{19}^+$. One can easily see
  that $\Aut(\mf h)$ has at least $8$ connected components since
  $\varphi_1, \varphi_2, \varphi_3$ defined as in \eqref{eq:2} give a
  subgroup of $D$ isomorphic to $(\bb Z_2)^3$. Moreover, define
  $\varphi_4 \in \Aut(\mf h)$ by
  \begin{align*}
    \varphi_4(e_1) = e_3, && \varphi_4(e_2) = e_2, && \varphi_4(e_3) = e_1. 
  \end{align*}
  Now $\varphi_1, \ldots, \varphi_4$ generate a subgroup of order
  $16$. Let us prove that this subgroup is isomorphic to $D$. In fact,
  note that if $\im(\ad_X) \cap Z(\mf h) = 0$, then
  $\alpha_1 = \alpha_3 = 0$. This implies
  $e^1(\varphi(e_2)) = e^3(\varphi(e_2)) = 0$ and so
  $\varphi(e_2) = \gamma_2 e_2 + \gamma_4 e_4 + \gamma_5 e_5 +
  \gamma_6 e_6$. If
  $\varphi(e_1) = \beta_1 e_1 + \beta_2 e_2 + \beta_3 e_3 + \beta_4
  e_4 + \beta_5 e_5 + \beta_6 e_6$ then $\beta_1 \neq 0$ or
  $\beta_3 \neq 0$. Moreover, since $\rank(\ad_{e_1}) = 2$ then
  $\beta_2 = 0$. Now, since $0 = [e_1, e_5] = [e_1, [e_1, e_2]]$ we
  have
  \begin{equation*}
    0 = [\varphi(e_1), [\varphi(e_1), \varphi(e_2)]] = -2 \beta_1 \beta_3 \gamma_2 e_6.
  \end{equation*}
  So $\beta_1 \neq 0$ if and only if $\beta_3 = 0$. The above
  computations show that $D$ equals the subgroup generated by
  $\varphi_1, \ldots, \varphi_4$. Denote $a = \varphi_4$,
  $b = \varphi_2$ and $c = \varphi_1$. Recall that
  $\varphi_3 = \varphi_1 \circ (\varphi_4 \circ \varphi_1)^2$. Then it
  is not hard to see that
  \begin{equation*}
    D \simeq \langle a, b, c \mid a^2 = b^2 = c^2 = (ac)^4 = e, ab = ba, bc = cb \rangle \simeq {\Dih_4} \times \bb Z_2.
  \end{equation*}
  
  \par \textbf{Case 3: $s = 4$.} The only possibility here is
  $\mf h = \mf h_9$. This case was already treated in
  \cite{reggianiModuliSpaceLeftinvariant2020}.
\end{proof}

\begin{remark}
  \par Supporting calculations for the proof of Theorem
  \ref{sec:full-autom-groups-1} can be found in \cite[Notebook 02]{cardosoAuxiliarySageMathNotebooks2024} from
  the GitHub repository accompanying this paper. For the sake of
  completeness we have also included the computations for the case
  $\mf h = \mf h_9$.
\end{remark}

\par It follows from Theorem \ref{sec:full-autom-groups-1} that the
full automorphism group of $\mf h$ can be computed essentially by
taking exponential of derivations (and making appropriate changes of
sign in the diagonal after that). Since this does not give a nice
expression for a generic automorphism, we need to develop a
simplification algorithm which works as follows. Let us write
\begin{equation*}
  \Der(\mf h) = \Der_{\mf n}(\mf h) \oplus
  \Der_{\mf d}(\mf h) \qquad \text{(direct sum of subspaces),}
\end{equation*}
where $\Der_{\mf n}(\mf h)$ consists of all the nilpotent
derivations of $\mf h$. Then we choose a basis $X_1, \ldots, X_n$ of
$\Der_{\mf n}(\mf h)$ and a basis $Y_1, \ldots, Y_m$ of
$\Der_{\mf d}(\mf h)$. In order to get better results, one
must take sparse matrices for such basis. We then 
\begin{itemize}
\item take the exponential of a generic element $X = \sum a_i X_i$ and
  simplify the expression for $e^X$;
\item take the exponential of $b_j Y_j$ (separately) and simplify the
  expression for $e^{b_j Y_j}$;
\item for every $k = 1, \ldots, j$ define
  $\varphi_k = \varphi_{k-1} e^{b_k Y_k}$, where $\varphi_0 = e^{X}$,
  and simplify the expression for $\varphi _k$.
\end{itemize}

\par Note that a good simplification performed in every step is
crucial and has the objective of preserving the original position of
the coefficients $a_i$ and $b_j$ in the final expression for a generic
automorphism. In Tables \ref{tab:aut-group-3-step},
\ref{tab:aut-group-4-step} and \ref{tab:aut-group-5-step} we present
the full automorphism group for every CSLAT of dimension $6$. In the
second column of these tables we give a description of the connected
component $\Aut_0(\mf h)$ and in the third column we give the
generators of the finite subgroup $D$ of $\Aut(\mf h)$ such that
$\Aut(\mf h) = \Aut_0(\mf h) \rtimes D$. Notice that we have renamed
the free parameters to $a_0, a_1, a_2, \ldots$

\section{The moduli space of left-invariant metrics for  CSLAT}\label{sec:moduli-space-left-1}

\par Let $H$ be a simply connected Lie group with Lie algebra $\mf
h$. Every left-invariant metric on $H$ is uniquely determined by an
inner product on $\mf h$. So, after a choice of a basis for $\mf h$ we
can identify the space $\mathcal M(H)$ of all left-invariant metrics
on $H$ with the symmetric space $\Sym_n^+ = \GL_n(\bb R) / {\OO(n)}$
of all symmetric positive definite $n \times n$ matrices. If we
identify $\Aut(\mf h) \subset \GL_n(\bb R)$, then $\Aut(\mf h)$ acts
on $\Sym_n^+$ on the right by
\begin{equation*}
  g \cdot \varphi = \varphi^T \, g \, \varphi, \qquad g \in \Sym_n^+, \varphi \in \Aut(\mf h).
\end{equation*}
We say that two left-invariant metrics on $H$ are \emph{equivalent} if they lie in the same orbit under the action of $\Aut(\mf h)$ (with the above identifications). The \emph{moduli space of left-invariant metrics up to isometric automorphism} is the quotient
\begin{equation*}
  \mathcal M(H) / {\sim} = \Aut(\mf h) \backslash \Sym_n^+.
\end{equation*}

We will also consider the moduli space of left-invariant metrics up to an isometric automorphism in the connected component of the identity
\begin{equation*}
  \mathcal M(H) / {\sim_0} = \Aut_0(\mf h) \backslash \Sym_n^+.
\end{equation*}

\par Assume that $\mf h$ is a CSLAT of dimension $6$. Let $\T_6^+$ be
the connected component of the identity of $\T_6$. Since $\T_6^+$ acts simply
transitively on $\Sym_6^+$, the action of $\Aut_0(\mf h)$ on $\Sym_6^+$
is equivalent to the left action of $\Aut_0(\mf h)$ on $\T_6^+$.

\begin{theorem}
  \label{sec:moduli-space-left}
  \it Let $\mf h$ be a CSLAT of dimension $6$ and $H$ be the simply
  connected Lie group with Lie algebra $\mf h$. Then
  $\mathcal M(H) / {\sim_0}$ is a smooth manifold. Moreover,
  $\mathcal M(H) / {\sim_0}$ is diffeomorphic to a submanifold
  $\Sigma$ of $\T_6^+$ which is transversal to $\Aut_0(\mathfrak
  h)$. The explicit description of the submanifolds $\Sigma$ can be
  found in Tables~\ref{tab:sigmas-3step-no-h19},
  \ref{tab:sigmas-4step-no-h26}, \ref{tab:sigmas-5step} and
  \ref{tab:sigmas4}.
\end{theorem}

\begin{proof}
  The same proof given in \cite{reggianiModuliSpaceLeftinvariant2020}
  for $\mf h_9$ can be adapted to the rest of the CSLATs. Let us recall
  how the submanifold $\Sigma$ is defined. Let
  $\varphi \in \Aut_0(\mf h)$ be a generic element given as in Tables
  \ref{tab:aut-group-3-step}, \ref{tab:aut-group-4-step} and
  \ref{tab:aut-group-5-step}. First we identify the places where the
  free parameters $a_0, a_1, a_2, \ldots$ occur and set the
  corresponding non-diagonal entries of the elements of $\Sigma \subset \T_6^+$ to $0$ and
  the diagonal ones to $1$. We fill the remaining places with free
  parameters $s_0, s_1, s_2, \ldots$ where $s_i \in \bb R$ for
  non-diagonal entries and $s_i > 0$ for the diagonal ones. The map $\sigma \in \Sigma \mapsto [\sigma^T \sigma]_0 \in \mathcal M(H) / {\sim_0}$ is the desired diffeomorphism.
\end{proof}

\par The algorithm for computing the $\Sigma$'s can be found in \cite[Notebook 00]{cardosoAuxiliarySageMathNotebooks2024} from our GitHub repository. We also define for future use the set $\mathrm{nd}(\Sigma)$ of non-diagonal free parameters of $\Sigma$.

\par Now we give a general description of $\mathcal M(H) /
{\sim}$. For each $g \in \mathcal M(H)$, we denote the orbits of $g$
under $\Aut(\mf h)$ and $\Aut_0(\mf h)$ by $[g]$ and $[g]_0$
respectively. Since $\Aut(\mf h) = \Aut_0(\mf h) \rtimes D$, we have
\begin{equation*}
  [g] = \bigcup_{\delta \in D} [\delta^T g \delta]_0.
\end{equation*}
Hence, every $\Aut(\mf h)$-orbit is a finite union of
$\Aut_0(\mf h)$-orbits. This does not mean that $D$ acts naturally on
$\Sigma$.

\par However, for the CSLATs with $\Aut(\mf h) \subset \T_6$, we have
$D (\Sigma) \subset \Sigma$. Moreover, $D$ acts on $\Sigma$ as a
finite subgroup of reflections of $\bb R^{n+1}$, where
$s_0, s_1, \ldots, s_n$ is the parametrization of $\Sigma$ given in
the proof of Theorem~\ref{sec:moduli-space-left}. Hence, there exists
a submanifold $\Sigma_D \subset \Sigma$ where $D$ acts trivially.  For
$\mf h_{19}^+$ and $\mf h_{26}^-$, we have $D \subset \OO(6)$ and
there exist a non-trivial submanifold $\Sigma_D \subset \Sigma$ as
before where $D$ acts trivially, although $D$ is no longer a subgroup
of reflections of $\Sigma$. For $\mf h_{13}$ we have
$D \not\subset \OO(6)$ and one can easily see that
$\Sigma_D = \varnothing$. In some cases, $\Sigma_D$ completely
determines the structure of $\mathcal M(H) / {\sim}$, but in general
this information is not enough. Supporting computations for these
facts and for the following examples are given in the attached notebook \cite[Notebook 03]{cardosoAuxiliarySageMathNotebooks2024}.

\begin{example}
  Let $\mf h = \mf h_{24}$. Recall that $D \simeq \bb Z_2$. Here we
  have that $\Sigma_D$ is given by elements of the form
  \begin{equation*}
    \begin{psmallmatrix}
      1 & 0 & 0 & 0 & 0 & 0 \\
      0 & s_{0} & 0 & 0 & 0 & 0 \\
      0 & 0 & s_{1} & 0 & 0 & 0 \\
      0 & 0 & s_{2} & s_{3} & 0 & 0 \\
      0 & 0 & 0 & 0 & s_{6} & 0 \\
      0 & 0 & 0 & s_{7} & 0 & s_{9}
    \end{psmallmatrix}
  \end{equation*}
  and $[g] = [g]_0$ if and only if $g = \sigma^T \sigma$ for some
  $\sigma \in \Sigma_D$.
\end{example}

\begin{example}
  Let $\mf h = \mf h_{11}$. In this case $D \simeq \bb Z_2 \oplus \bb
  Z_2$ is generated by the automorphisms $\varphi_1, \varphi_2$ given
  in the proof of Theorem~\ref{sec:full-autom-groups-1}. If we denote
  by $\Sigma_{\varphi_i}$ the submanifold of $\Sigma$ where
  $\varphi_i$ acts trivially, then
  \begin{equation*}
    \Sigma_D = \Sigma_{\varphi_1} \cap \Sigma_{\varphi_2}.
  \end{equation*}
  For a given left-invariant metric $g$, we have that $[g]$ is the
  union of $1$, $2$, or $4$ $\Aut_0(\mf h)$-orbits if and only if
  $[g]_0 = [\sigma^T \sigma]_0$, for $\sigma \in \Sigma_D$,
  $\sigma \in \Sigma_{\varphi_1} \triangle\, \Sigma_{\varphi_2}$ or
  $\sigma \in \Sigma - (\Sigma_{\varphi_1} \cup \Sigma_{\varphi_2})$,
  respectively.
\end{example}

\begin{example}
  For $\mf h = \mf h_{19}^+$ we have $\Sigma_D$ consists of matrices
  of the form
  \begin{equation*}
    \begin{psmallmatrix}
      1 & 0 & 0 & 0 & 0 & 0 \\
      0 & 1 & 0 & 0 & 0 & 0 \\
      0 & 0 & 1 & 0 & 0 & 0 \\
      0 & 0 & 0 & s_{3} & 0 & 0 \\
      0 & 0 & 0 & 0 & s_{3} & 0 \\
      0 & 0 & 0 & 0 & 0 & s_{9}
    \end{psmallmatrix}
  \end{equation*}
  and for $\mf h = \mf h_{26}^-$ we have that $\Sigma_D$ consists of
  matrices of the form
  \begin{equation*}
    \begin{psmallmatrix}
      1 & 0 & 0 & 0 & 0 & 0 \\
      0 & 1 & 0 & 0 & 0 & 0 \\
      0 & 0 & s_{1} & 0 & 0 & 0 \\
      0 & 0 & 0 & s_{3} & 0 & 0 \\
      0 & 0 & 0 & 0 & s_{3} & 0 \\
      0 & 0 & 0 & 0 & 0 & s_{10}
    \end{psmallmatrix}.
  \end{equation*}
\end{example}

\begin{example}
  Let $\mf h = \mf h_{13}$. Consider the metric
  \begin{equation*}
    g_1 =
    \begin{psmallmatrix}
      2 & 1 & 0 & 0 & 0 & 0 \\
      1 & 1 & 0 & 0 & 0 & 0 \\
      0 & 0 & 1 & 0 & 0 & 0 \\
      0 & 0 & 0 & 1 & 0 & 0 \\
      0 & 0 & 0 & 0 & 1 & 0 \\
      0 & 0 & 0 & 0 & 0 & 1
    \end{psmallmatrix}
    = \sigma_1^T \sigma_1 \text{ with } \sigma_1 =
    \begin{psmallmatrix}
      1 & 0 & 0 & 0 & 0 & 0 \\
      1 & 1 & 0 & 0 & 0 & 0 \\
      0 & 0 & 1 & 0 & 0 & 0 \\
      0 & 0 & 0 & 1 & 0 & 0 \\
      0 & 0 & 0 & 0 & 1 & 0 \\
      0 & 0 & 0 & 0 & 0 & 1
    \end{psmallmatrix}
    \in \Sigma.
  \end{equation*}
  Now let $g_2 = \varphi_3^T g_1 \varphi_3$ where $\varphi_3$ is given
  as in the proof of Theorem \ref{sec:full-autom-groups-1}. Then,
  there does not exist $\sigma \in \Sigma$ such that $g_2 = \sigma^T
  \sigma$. However, $g_2 \in [g_3]_0$ where
  \begin{equation*}
    g_3 =
    \begin{psmallmatrix*}[r]
      2 & -1 & 0 & 0 & 0 & 0 \\
      -1 & 1 & 0 & 0 & 0 & 0 \\
      0 & 0 & 2 & 1 & 0 & 0 \\
      0 & 0 & 1 & 1 & 0 & 0 \\
      0 & 0 & 0 & 0 & 2 & 0 \\
      0 & 0 & 0 & 0 & 0 & \frac{1}{2}
    \end{psmallmatrix*}
    = \sigma_3^T \sigma_3 \text{ with } \sigma_3 =
    \begin{psmallmatrix*}[r]
      1 & 0 & 0 & 0 & 0 & 0 \\
      -1 & 1 & 0 & 0 & 0 & 0 \\
      0 & 0 & 1 & 0 & 0 & 0 \\
      0 & 0 & 1 & 1 & 0 & 0 \\
      0 & 0 & 0 & 0 & \sqrt{2} & 0 \\
      0 & 0 & 0 & 0 & 0 & \frac{1}{2} \, \sqrt{2}
    \end{psmallmatrix*}
    \in \Sigma.
  \end{equation*}
\end{example}

\section{The full isometry group of a CSLAT}
\label{sec:full-isometry-group-1}

\par Let $\mf h$ be a CSLAT of dimension $6$ and let $H$ be the
associated simply connected Lie group. We identify
$\mathcal M(H) / {\sim_0}$ with $\Sigma \subset \T_6^+$ as in Theorem
\ref{sec:moduli-space-left}. Given $\sigma \in \Sigma$, define
\begin{equation*}
  g_\sigma = \sigma^T \sigma \in \Sym_6^+.
\end{equation*}
We identify as usual $g_\sigma$ with an inner product on $\mf h$ and
the corresponding left-invariant metric on $H$. In this section we
address the problem of computing the full isometry group
$\I(H, g_\sigma)$ of $g_\sigma$. According to
\cite{wolfLocallySymmetricSpaces1963} we know that
\begin{equation*}
  \I(H, g_\sigma) \simeq H \rtimes K, 
\end{equation*}
where $H$ is identified as the subgroup of left translations of $H$ and
$K$ is the full isotropy group. Moreover, via the isotropy
representation we have the isomorphism
\begin{equation*}
  K \simeq \Aut(\mf h) \cap \OO(g_\sigma).
\end{equation*}
Recall from Remark \ref{sec:full-autom-groups-2} that
$\Aut(\mf h) = \Aut_0(\mf h) \rtimes D$, where $D$ is described in
Theorem \ref{sec:full-autom-groups-1}.

\begin{theorem}
  \label{sec:full-isometry-group} 
  \it We keep the notation of this section. Assume that $\mf h$ is not isomorphic to $\mf h_{13}, \mf h_{19}^+$ nor $\mf h_{26}^-$. Then $K \subset D$. 
\end{theorem}

\begin{proof}
  Notice that in these cases, under the usual identifications,
  $\Aut(\mf h) \subset \T_6$ and $D \simeq \bb Z_2^k$ for some
  $k \in \{1, 2, 3\}$. Let
  $\varphi \in K = \Aut(\mf h) \cap \OO(g_\sigma)$ and let
  $\varphi = S + N$ its Chevalley decomposition, i.e., $N$ is
  nilpotent, $S$ is semisimple, $[N, S] = 0$ and $N$, $S$ are
  polynomials in $\varphi$. Then $\varphi$ and $S$ have the same
  diagonal, and hence the elements on the diagonal are $\pm 1$. We
  claim that $N = 0$. Indeed, first we assume that
  $\diag \varphi = (1, \ldots, 1)$. So, $\varphi = N$ and the set
  $\{\varphi^n: n \in \mathbb N\}$ is unbounded in $\Aut(\mf h)$,
  which contradicts the fact that $\varphi \in \OO(g_\sigma)$. In the
  general case we have that $\diag \varphi^2 = \diag S^2 = (1, \ldots,
  1)$ and so the nilpotent part of $\varphi^2$ is trivial. But
  $\varphi^2 = (S + N)^2 = S^2 + 2SN + N^2$ and thus $N^2 = -2SN$ which
  is impossible unless $N = 0$.

  Now we have that $\varphi = S \in K$ is a semisimple automorphism
  such that $\varphi^2 = \id$, which satisfy the isometry condition
  \begin{equation}\label{eq:3}
    \varphi^T g_\sigma = g_\sigma \varphi.
  \end{equation}
  Observe that equation \eqref{eq:3} is linear in the coefficients of
  $\varphi$ (although it is not necessarily linear in the free
  parameters $a_0, a_1, a_2, \ldots$ describing $\varphi$). Also
  notice that if a free parameter $a_\ell$ appears for the first time
  (with the lexicographic order) in the entry $\varphi_{ij}$, then
  $\sigma_{ij} = 0$. From these facts it is not hard to conclude that
  all the entries below the diagonal of $\varphi$ are zero. Hence
  $\varphi \in D$. One can also prove this by direct inspection. Such
  computations are easy (but long and tedious) and are provided in the
  accompanying notebook \cite[Notebook 04]{cardosoAuxiliarySageMathNotebooks2024}.
\end{proof}

\begin{remark}
  \par From Theorem \ref{sec:full-isometry-group}, one can easily
  compute the full isometry group of any $g_\sigma$ for all the
  $\mf h$ such that $\Aut(\mf h) \subset \T_6$. Such description is
  rather involved and it is not convenient to be included in the
  manuscript. However, we do include in Tables
  \ref{tab:sigmas-3step-no-h19}, \ref{tab:sigmas-4step-no-h26} and
  \ref{tab:sigmas-5step} some relevant information on the metrics with
  nontrivial subgroup of isometric automorphisms. These tables must be
  read as follows. The first and second columns present the name of
  the Lie algebra and a generic representative of every left-invariant
  metric up to isometric automorphism (actually, here we keep it
  simple by considering $\mathcal M(H) / {\sim_0}$ instead of
  $\mathcal M(H) / {\sim}$). For each generic metric $g_\sigma$, one
  can set equal to $0$ a number $p$ of non diagonal parameters. We
  indicate this in the third column with
  $\# \operatorname{nd}(\Sigma) = 0$. That is, if
  $\operatorname{nd}(\Sigma)$ has $n$ elements, then there are $\binom np$
  possible such substitutions. The last columns are for the
  isomorphism classes of subgroups of $D$ (recall that $K$ is a
  subgroup of $D$) and we present the amount of substitutions which
  have $K$ isomorphic to one of these groups.

  \par For example, if $\mf h = \mf h_{12}$, then
  $\operatorname{nd}(\Sigma) = \{s_0, s_1, s_3, s_5, s_6\}$ and we
  have $\binom 52 = 10$ ways to choose two parameters equal to
  $0$. For $8$ of these choices we have $K = \{e\}$ and for the
  remaining $2$ we have $K \simeq \mathbb Z_2$. The explicit
  substitutions and subgroups $K \hookrightarrow D$ can be found in
  \cite[Notebook 05]{cardosoAuxiliarySageMathNotebooks2024}.
\end{remark}

\begin{remark}
  \par According to Theorem \ref{sec:full-isometry-group}, if
  $\mf h \not\simeq \mf h_{13}, \mf h_{19}^+, \mf h_{26}^-$, there
  exists a finite number of subgroups of $\Aut(\mf h)$ (and moreover,
  of $D$) that can be realized as the full isotropy group of
  $\I(H, g_\sigma)$ for any $\sigma \in \Sigma$. This is not true if
  $\mf h = \mf h_{13}, \mf h_{19}^+$ or $\mf h_{26}^-$. In fact, if $r
  \in (0, 1]$ then the automorphisms
  \begin{align*}
    {\varphi_r =
    \begin{psmallmatrix*}[r]
      0 & -r & 0 & 0 & 0 & 0 \\
      -\frac{1}{r} & 0 & 0 & 0 & 0 & 0 \\
      0 & 0 & 1 & 0 & 0 & 0 \\
      0 & 0 & -1 & -1 & 0 & 0 \\
      0 & 0 & 0 & 0 & 0 & r \\
      0 & 0 & 0 & 0 & \frac{1}{r} & 0
    \end{psmallmatrix*}},
    &&
       \varphi_r' =
       \begin{psmallmatrix*}[r]
         0 & 0 & r & 0 & 0 & 0 \\
         0 & 1 & 0 & 0 & 0 & 0 \\
         \frac{1}{r} & 0 & 0 & 0 & 0 & 0 \\
         0 & 0 & 0 & 0 & \frac{1}{r} & 0 \\
         0 & 0 & 0 & r & 0 & 0 \\
         0 & 0 & 0 & 0 & 0 & 1
       \end{psmallmatrix*},
    &&
       \varphi_r'' =
       \begin{psmallmatrix*}[r]
         0 & r & 0 & 0 & 0 & 0 \\
         \frac{1}{r} & 0 & 0 & 0 & 0 & 0 \\
         0 & 0 & -1 & 0 & 0 & 0 \\
         0 & 0 & 0 & 0 & -r & 0 \\
         0 & 0 & 0 & -\frac{1}{r} & 0 & 0 \\
         0 & 0 & 0 & 0 & 0 & -1
       \end{psmallmatrix*}
  \end{align*}
  of $\mf h_{13}$, $\mf h_{19}^+$ and $\mf h_{26}^-$ respectively are
  isometric with respect to $g_{\sigma_r}$, $g_{\sigma'_r}$ and
  $g_{\sigma_r''}$ where
  \begin{align*}
    {\sigma_r =
    \begin{psmallmatrix*}[r]
      1 & 0 & 0 & 0 & 0 & 0 \\
      \frac{\sqrt{-r^{2} + 1}}{r} & 1 & 0 & 0 & 0 & 0 \\
      0 & 0 & 1 & 0 & 0 & 0 \\
      0 & 0 & \frac{1}{2} & 1 & 0 & 0 \\
      0 & 0 & 0 & 0 & \frac{1}{r} & 0 \\
      0 & 0 & 0 & 0 & 0 & 1
    \end{psmallmatrix*}},
    &&
    \sigma_r' =
       \begin{psmallmatrix*}[r]
         1 & 0 & 0 & 0 & 0 & 0 \\
         \frac{\sqrt{-r^{2} + 1}}{r} & 1 & 0 & 0 & 0 & 0 \\
         0 & \sqrt{-r^{2} + 1} & 1 & 0 & 0 & 0 \\
         0 & 0 & 0 & r^{2} & 0 & 0 \\
         0 & 0 & 0 & 0 & r & 0 \\
         0 & 0 & 0 & 0 & 0 & 1
       \end{psmallmatrix*},
 &&
       \sigma_r'' =
       \begin{psmallmatrix*}[r]
         1 & 0 & 0 & 0 & 0 & 0 \\
         \frac{\sqrt{-r^{2} + 1}}{r} & 1 & 0 & 0 & 0 & 0 \\
         0 & 0 & 1 & 0 & 0 & 0 \\
         0 & 0 & 0 & \frac{1}{r} & 0 & 0 \\
         0 & 0 & 0 & 0 & 1 & 0 \\
         0 & 0 & 0 & 0 & 0 & 1
       \end{psmallmatrix*}.
  \end{align*}
  See \cite[Notebook 06]{cardosoAuxiliarySageMathNotebooks2024} for verification. However, in each of these cases the
  subgroup $K$ is conjugated to a subgroup of $D$, since $D$ is a
  maximal compact subgroup of $\Aut(\mf h)$.
\end{remark}

\section{The index of symmetry for CSLAT}
\label{sec:index-symmetry-csla-title}

\par This section concerns the computation of the index of symmetry of
a left-invariant metric on a Lie group associated with a CSLAT. We
refer to \cite{olmosIndexSymmetryCompact2014} and
\cite{berndtCompactHomogeneousRiemannian2017} for the general
definitions and the structure theory related with the index of
symmetry of a general homogeneous Riemannian manifold.

Let $\mathfrak h$ be a CSLAT of dimension $6$ and let $H$ be the
associated simply connected Lie group. Endow $H$ with a left-invariant
Riemannian metric $g$ and denote by $\I(H)$ its full isometry
group. Since the isotropy subgroup of $\I(H)$ is discrete, every
Killing field on $H$ is right-invariant. Let $X, Y, Z \in \mathfrak h$
and denote by $X^*, Y^*, Z^*$ the corresponding right-invariant vector
fields. By definition, the \emph{distribution of symmetry} at the
identity $e$ is given by
\begin{equation*}
  \mathfrak s_e = \{X_e: X \in \mathfrak h \text{ and } (\nabla X^*)_e = 0\}.
\end{equation*}

Since $\mathfrak s$ is $\I(H)$-invariant, it coincides with the
distribution generated by the left-invariant fields
$X \in \mathfrak h$ such that $X_e \in \mathfrak s_e$. The so-called
\emph{index of symmetry} of $(H, g)$ is the rank of $\mf s$.  Recall
the well-known Kozsul formula for Killing fields: for all
right-invariant fields $X^*, Y^*, Z^*$, we have
\begin{equation*}
  g(\nabla_{X^*}Y^*, Z^*) =  \frac12 (g([X^*, Y^*], Z^*) + g([X^*, Z^*], Y^*) + g([Y^*, Z^*], X^*)).
\end{equation*}
So, $Y \in \mathfrak s$ if and only if for all $X, Z \in \mathfrak h$
it holds that
\begin{align*}
0 & = 2 g((\nabla_{X^*}Y^*)_e, (Z^*)_e) \\
& =  g([X^*, Y^*]_e, (Z^*)_e) + g([X^*, Z^*]_e, (Y^*)_e) + g([Y^*, Z^*]_e, (X^*)_e) \\
& = -g([X, Y]_e, Z_e) - g([X, Z]_e, Y_e) - g([Y, Z]_e, X_e),
\end{align*}
since $[X^*, Z^*]_e = -[X, Z]^*_e$. So, the above equation is
equivalent to
\begin{equation}\label{eq:6}
  g([X, Y], Z) + g([X, Z], Y) + g([Y, Z], X) = 0
\end{equation}
for all $X, Z \in \mathfrak h$. Notice that we have identified with
the same symbol $g$ the Riemannian metric on $H$ and the corresponding
inner product on $\mathfrak h$. If $e_1, \ldots, e_6$ is the standard
basis of $\mathfrak h$ and we write $Y = \sum y_i e_i$, then the
previous equation is linear in $y_1, \ldots, y_6$. This equation is
not linear if we also consider the metric coefficients. Since every
$g$ is isometric to a metric given by $g_\sigma = \sigma^T
\sigma$, for some $\sigma \in \Sigma$, in order to determine the
existence of metrics with non-trivial index of symmetry it is enough
to work with the metrics of this form.

\begin{theorem}\label{sec:index-symmetry-csla}
  \it Let $\mf h$ be a CSLAT of dimension $6$ and let $H$ be its associated   simply connected Lie group. Assume that $\mf h$ is not isomorphic to $\mf h_9, \mf h_{10}, \mf h_{21}, \mf h_{22}$ nor $\mathfrak h_{28}$. Then every left-invariant metric on $H$ has trivial index of symmetry.
\end{theorem}

As far as we know, these are the first examples in the literature of a
Lie group all of whose left-invariant metrics have trivial index of
symmetry. Now we direct our attention to the Lie groups with
characteristically solvable Lie algebra (of triangular type) which admit left-invariant
metrics with nontrivial index of symmetry. We use the notation of
Section \ref{sec:moduli-space-left-1} for the left-invariant metrics
under study. We begin with the cases where the distribution of symmetry
is contained in the center of the Lie algebra.

\begin{theorem}\label{sec:index-symmetry-csla-1}
  \it The left-invariant metric $g_\sigma = \sigma ^T \sigma$, $\sigma \in  \Sigma$, on $H_9$ has nontrivial index of symmetry if and only if $s_2 = 0$. In this case, the index of symmetry is $1$ and the distribution of symmetry is generated by the left-invariant field $Y = e_4$. In particular, $\mf s$ is properly contained in the central distribution of $H_9$. 
\end{theorem}

\begin{theorem}\label{sec:index-symmetry-csla-2}
  \it The left-invariant metric $g_\sigma = \sigma ^T \sigma$, $\sigma \in \Sigma$ on $H_{22}$, has nontrivial index of symmetry if and only if $s_1 = s_3 = 0$. In this case, the index of symmetry is $1$ and the distribution of symmetry is generated by the left-invariant field $Y = e_3$. In particular, $\mf s$ is properly contained in the central distribution of $H_{22}$.
\end{theorem}

\begin{theorem}\label{sec:index-symmetry-csla-3}
  \it The left-invariant metric $g_\sigma = \sigma ^T \sigma$, $\sigma \in \Sigma$, on $H_{10}$ has nontrivial index of symmetry if and only if $s_{3} = \frac{s_{1} s_{4}}{s_{2}}$. In this case the index of symmetry is $1$, and the distribution of symmetry is generated by the left-invariant field
  $$Y = e_2 -\frac{s_{1}}{s_{2}}e_3 + \frac{s_{0}^{2}
    s_{4}}{s_{2}^{2} s_{5}}e_5 -\frac{s_{0}^{2} s_{2}^{2} +
    s_{0}^{2} s_{4}^{2}}{s_{2}^{2} s_{5}^{2}}e_6
  $$
  which does not belong to the center of $\mathfrak h_{10}$.
\end{theorem}

\begin{proof}[{\proofname} of Theorems \ref{sec:index-symmetry-csla},
  \ref{sec:index-symmetry-csla-1}, \ref{sec:index-symmetry-csla-2} and
  \ref{sec:index-symmetry-csla-3}]
  It follows from very long, but straightforward, computations. See
  the attached notebook \cite[Notebook 07]{cardosoAuxiliarySageMathNotebooks2024}.
\end{proof}

The last two cases are more difficult to deal with, as the
computations are lengthy and also more involved. A case by case
analysis is needed to describe the general situation.

\begin{theorem}\label{sec:index-symmetry-csla-4}
  \it Let us consider the left-invariant metric $g_\sigma = \sigma^T
  \sigma$, $\sigma \in \Sigma$, on $H_{21}$. Then:
  \begin{enumerate}
  \item\label{item:thm-h21-1} if $s_2 \neq 0$, the metrics with
    nontrivial index of symmetry form an algebraic hypersurface of
    $\mathcal M(H_{21}) / {\sim_0}$. For all these metrics, we have
    $i_{\mathfrak s}(H_{21}, g_\sigma) = 1$ and
    $\mathfrak s \not\subset Z(\mf h_{21})$;
  \item\label{item:thm-h21-2} if $s_2=0, s_0\neq 0$ and $s_{6} = \frac{s_{3} s_{7}}{s_4}$, then
    $i_{\mathfrak s}(H_{21}, g_\sigma) = 1$ and
    $\mathfrak s \not\subset Z(\mf h_{21})$;
  \item\label{item:thm-h21-3} if
    $s_2=0, s_0 = 0$ and $s_6 \neq \frac{s_3 s_7}{s_4}$, the metrics with
    nontrivial index of symmetry form an algebraic submanifold of
    $\mathcal M(H_{21}) / {\sim_0}$ of codimension $2$. For these
    metrics we have $i_{\mathfrak s}(H_{21}, g_\sigma) = 2$ and
    $\mathfrak s \not\subset Z(\mf h_{21})$;
  \item\label{item:thm-h21-4} if $s_2=0, s_0 = 0, s_6 = \frac{s_3 s_7}{s_4}$
    and $s_5 \neq \frac{s_1^2 s_7}{s_4^2}$ then $i_{\mathfrak s}(H_{21}, g_\sigma) = 1$ and $\mathfrak s \subset \nobreak Z(\mf h_{21})$;
  \item\label{item:thm-h21-5} if $s_2=0, s_0 = 0, s_6 = \frac{s_3 s_7}{s_4}, 
    s_5 = \frac{s_1^2 s_7}{s_4^2}$ and $s_3 \neq 0$ then $i_{\mathfrak s}(H_{21}, g_\sigma) = 2$ and $\mathfrak s \not\subset Z(\mf h_{21})$;
  \item\label{item:thm-h21-6} if $s_2=0, s_0 = 0, s_6 = \frac{s_3 s_7}{s_4},   
    s_5 = \frac{s_1^2 s_7}{s_4^2}$ and $s_3 = 0$, then $i_{\mathfrak s}(H_{21},g_\sigma) = 3$ and $\mathfrak s \not\subset Z(\mf h_{21})$.\footnote{In this case condition $s_6 = \frac{s_3 s_7}{s_4}$ just means $s_6 = 0$.     We stated the theorem in this way in order to better visualize all the possible cases for $\sigma \in \Sigma$.}
  \item In all the remaining cases the we have $i_{\mf s}(H_{21}, g_\sigma) = 0$.
  \end{enumerate}
\end{theorem}

\begin{proof}
  See the attached notebook \cite[Notebook 08]{cardosoAuxiliarySageMathNotebooks2024} for the supporting computations and the following remark for a precise description of the distribution of
  symmetry.
\end{proof}

\begin{remark}
  We keep the hypothesis from Theorem \ref{sec:index-symmetry-csla-4}.
  \begin{enumerate}
  \item If $s_2 \neq 0$, the algebraic hypersurface of $\Sigma$ such that
    the associated left-invariant metrics have non-trivial index of
    symmetry is defined by
    \begin{align*}
      s_{0} s_{2} s_{4}^{2} s_{5} - s_{0}^{2} s_{1} s_{4} s_{6} - s_{0} s_{2} s_{3} s_{4} s_{6} - s_{0} s_{1}^{2} s_{2} s_{7} + s_{0}^{2} s_{1} s_{3} s_{7} - s_{1} s_{2}^{2} s_{3} s_{7} + s_{0} s_{2} s_{3}^{2} s_{7} &= 0.
    \end{align*}
    In order to describe a left-invariant field $Y$ generating
    $\mathfrak s$ we consider the case $s_0 \neq 0$, where
    \begin{align*}
      Y & = e_2 - \frac{s_{0}^{2} s_{1} s_{4} s_{6} - {\left(s_{0}^{2} s_{1} + s_{1} s_{2}^{2}\right)} s_{3} s_{7}}{s_{0} s_{2}^{2} s_{4} s_{7}} e_3 - \frac{s_{0} s_{1} + s_{2} s_{3}}{s_{2} s_{4}} e_4 -\frac{s_{0} s_{1}^{2} s_{2} - s_{0} s_{2} s_{3}^{2} - {\left(s_{0}^{2} s_{1} - s_{1} s_{2}^{2}\right)} s_{3}}{s_{0} s_{2} s_{4}^{2}} e_5 \\ 
      & \hspace{30pc} + \alpha e_6 \notin Z(\mf h_{21}) 
    \end{align*}
    where
    \begin{align*}
      \alpha & = \frac{s_{0} s_{1}^{3}}{s_{2} s_{4}^{3}} + \frac{2 \, s_{1}^{2} s_{3}}{s_{4}^{3}} - \frac{s_{0}^{2} s_{1}^{2} s_{3}}{s_{2}^{2} s_{4}^{3}} - \frac{2 \, s_{0} s_{1} s_{3}^{2}}{s_{2} s_{4}^{3}} + \frac{s_{1} s_{2} s_{3}^{2}}{s_{0} s_{4}^{3}} - \frac{s_{3}^{3}}{s_{4}^{3}} + \frac{s_{0} s_{1} s_{4}}{s_{2} s_{7}^{2}} + \frac{s_{1}^{2} s_{6}}{s_{4}^{2} s_{7}} + \frac{s_{0}^{2} s_{1}^{2} s_{6}}{s_{2}^{2} s_{4}^{2} s_{7}} + \frac{s_{0} s_{1} s_{3} s_{6}}{s_{2} s_{4}^{2} s_{7}} \\ 
      & \hspace{31pc} + \frac{s_{1} s_{2} s_{3} s_{6}}{s_{0} s_{4}^{2} s_{7}};
    \end{align*}
    and the case $s_0 = 0$, where
    \begin{equation*}
      Y = e_2 + \frac{s_{4}^{2} s_{5} - s_{1}^{2} s_{7}}{s_{2} s_{4}
        s_{7}} e_3 - \frac{s_{5}}{s_{7}} e_5 + \frac{s_{5} s_{6}}{s_{7}^{2}} e_6 \notin Z(\mf h_{21}).
    \end{equation*}
    
  \item If $s_2=0, s_0\neq 0, s_{6} = \frac{s_{3} s_{7}}{s_4}$, then $\mf s$ is generated by
    \begin{equation*}
      Y = \frac{s_{4}^{2} s_{5} - s_{1}^{2} s_{7}}{s_{0} s_{1} s_{7}} e_3 + e_4 - \frac{s_{3}}{s_{4}} e_5 - \frac{s_{4}^{4} + s_{4}^{2} s_{5} s_{7} - s_{3}^{2} s_{7}^{2}}{s_{4}^{2} s_{7}^{2}} e_6 \notin Z(\mf h_{21}).
    \end{equation*}
  \item When $s_2=0$, $s_0 = 0$ and $s_6 \neq \frac{s_3 s_7}{s_4}$,
    the submanifold of $\Sigma$ associated to the metrics with
    nontrivial index of symmetry is described by the equation
    \begin{align*}
      s_{4}^{4} s_{5}^{2} - 2 \, s_{3} s_{4}^{3} s_{5} s_{6} + s_{3}^{2} s_{4}^{2} s_{6}^{2} + {\left(s_{1}^{4} + 2 \, s_{1}^{2} s_{3}^{2} + s_{3}^{4}\right)} s_{7}^{2} - 2 \, {\left({\left(s_{1}^{2} - s_{3}^{2}\right)} s_{4}^{2} s_{5} + {\left(s_{1}^{2} s_{3} + s_{3}^{3}\right)} s_{4} s_{6}\right)} s_{7}  = 0
    \end{align*}
    then $\mf s$ is generated by $Y_1 = e_3 \in Z(\mf h_{21})$ and
    \begin{align*}
      Y_2 & = e_2 - \frac{s_{4}^{2} s_{5} + s_{3} s_{4} s_{6} - {\left(s_{1}^{2} + s_{3}^{2}\right)} s_{7}}{2 \, {\left(s_{4}^{2} s_{6} - s_{3} s_{4} s_{7}\right)}} e_4 - \frac{s_{4}^{2} s_{5} s_{6} - s_{3} s_{4} s_{6}^{2} - {\left(2 \, s_{3} s_{4} s_{5} - {\left(s_{1}^{2} + s_{3}^{2}\right)} s_{6}\right)} s_{7}}{2 \, {\left(s_{4}^{2} s_{6} s_{7} - s_{3} s_{4} s_{7}^{2}\right)}} e_5 \\ 
      & \hspace{29pc} + \alpha e_6 \notin Z(\mf h_{21})
    \end{align*}
    where
    \begin{align*}
      \alpha & = \frac{s_{4}^{4} s_{5} - s_{3} s_{4}^{3} s_{6} + s_{4}^{2} s_{5} s_{6}^{2} - s_{3} s_{4} s_{6}^{3} - {\left(s_{1}^{2} + s_{3}^{2}\right)} s_{5} s_{7}^{2} + {\left(s_{4}^{2} s_{5}^{2} - s_{3} s_{4} s_{5} s_{6} - {\left(s_{1}^{2} - s_{3}^{2}\right)} s_{4}^{2} + {\left(s_{1}^{2} + s_{3}^{2}\right)} s_{6}^{2}\right)} s_{7}}{2 \, s_{4}^{2} s_{6} s_{7}^{2} - 2 \, s_{3} s_{4} s_{7}^{3}} .
    \end{align*}
    
  \item If $s_2=0$, $s_0 = 0$, $s_6 = \frac{s_3 s_7}{s_4}$ and
    $s_5 \neq \frac{s_1^2 s_7}{s_4^2}$ then $\mf s$ is generated by
    $e_3 \in Z(\mf h_{21})$.
  \item If $s_2=0$, $s_0 = 0$, $s_6 = \frac{s_3 s_7}{s_4}$,
    $s_5 = \frac{s_1^2 s_7}{s_4^2}$ and $s_3 \neq 0$, then
    $\mathfrak s$ is generated by $Y_1 = e_3 \in Z(\mf h_{21})$ and
    \begin{equation*}
      Y_2 = e_4 - \frac{s_{3}}{s_{4}} e_5 - \frac{s_{4}^{4} + {\left(s_{1}^{2} - s_{3}^{2}\right)} s_{7}^{2}}{s_{4}^{2} s_{7}^{2}} e_6 \notin Z(\mf h_{21}).
    \end{equation*}
  \item If
    $s_2=0$, $s_0 = 0$, $s_6 = \frac{s_3 s_7}{s_4}$, $s_5 = \frac{s_1^2
      s_7}{s_4^2}$ and $s_3 = 0$, then $\mf s$ is generated by
    \begin{align*}
      Y_1 = e_{2} - \tfrac{s_{1}^{2}}{s_{4}^{2}}e_5 \notin Z(\mf h_{21}), && Y_2 = e_{3} \in Z(\mf h_{21}), && Y_3 = e_{4} - \frac{{\left(s_{4}^{4} + s_{1}^{2} s_{7}^{2}\right)}}{s_{4}^{2} s_{7}^{2}}e_6 \notin Z(\mf h_{21}).
    \end{align*}
  \end{enumerate}
\end{remark}

\begin{theorem}\label{sec:index-symmetry-csla-5}
  \it Let us consider the left-invariant metric $g_\sigma = \sigma^T
  \sigma$, $\sigma \in \Sigma$, on $H_{28}$ and let us define
  \begin{align*}
    A = 
    \begin{pmatrix}
      a & b & \alpha \\
      c & \beta & b \\
      \gamma & c & a
    \end{pmatrix} && \text{where} &&
                                   \begin{matrix*}[l]
                                     a = s_{5}^{2} s_{6} + s_{3} s_{5} s_{8} - {\left(s_{1} s_{2} + s_{3} s_{4}\right)} s_{9}, \\
                                     b = s_{5}^{2} s_{7} + s_{4} s_{5} s_{8} - {\left(s_{2}^{2} + s_{4}^{2} + s_{3} s_{5}\right)} s_{9}, \\
                                     c = s_{4} s_{5} s_{6} + s_{3} s_{5} s_{7} - {\left(s_{0}^{2} + s_{1}^{2} + s_{3}^{2}\right)} s_{9}, \\
                                     \alpha = 2 \, s_{5}^{2} s_{8} - 2 \, s_{4} s_{5} s_{9}, \\
                                     \beta = 2 \, s_{4} s_{5} s_{7} - 2 \, {\left(s_{1} s_{2} + s_{3} s_{4}\right)} s_{9}, \\
                                     \gamma = 2 \, s_{3} s_{5} s_{6}.
                                   \end{matrix*}
  \end{align*}
  Then
  \begin{equation*}
    i_{\mathfrak s}(H_{28}, g_\sigma)  = 3 - \rank A.
  \end{equation*}
\end{theorem}

\begin{proof}
  See the attached notebook \cite[Notebook 09]{cardosoAuxiliarySageMathNotebooks2024}.
\end{proof}

\begin{remark}
  The matrix $A$ in Theorem \ref{sec:index-symmetry-csla-5} is a
  persymmetric matrix (i.e., it is symmetric with respect to its
  anti-diagonal). However this property has not a geometric or
  algebraic meaning as its form depends heavily on the basis of choice
  and even in the given order of the algorithm solving the equations
  for metrics with non-trivial index of symmetry.  We present in Table~\ref{tab:h28-examples-index} examples of different metrics
  $g_\sigma$, $\sigma \in \Sigma$, for which $A$ has all the possible
  ranks. Notice that in these particular examples,
  $i_{\mf s}(H_{28}, g_\sigma) > 0$ implies $\mf s \not\subset Z(\mf
  h_{28})$. Moreover, in the following result we prove that non-trivial
  distributions of symmetry are never central.
\end{remark}

\begin{table}
  \caption{Examples of left-invariant metrics on $H_{28}$ with all the
    possible indexes of symmetry.}
  \centering
  \(    
  \begin{array}{|c|c|c|l|}
    \hline
    \sigma & g_\sigma & A & \mf s \\    
    \hline\hline 
    \begin{psmallmatrix*}[r]
      1 & 0 & 0 & 0 & 0 & 0 \\
      0 & 1 & 0 & 0 & 0 & 0 \\
      0 & 0 & 1 & 0 & 0 & 0 \\
      0 & 0 & 0 & 1 & 0 & 0 \\
      0 & 0 & 0 & 0 & 1 & 0 \\
      0 & 0 & 0 & 0 & \frac{1}{2} & 1
    \end{psmallmatrix*}
    &
      \begin{psmallmatrix*}[r]
        1 & 0 & 0 & 0 & 0 & 0 \\
        0 & 1 & 0 & 0 & 0 & 0 \\
        0 & 0 & 1 & 0 & 0 & 0 \\
        0 & 0 & 0 & 1 & 0 & 0 \\
        0 & 0 & 0 & 0 & \frac{5}{4} & \frac{1}{2} \\
        0 & 0 & 0 & 0 & \frac{1}{2} & 1
      \end{psmallmatrix*}
    &
      \begin{psmallmatrix*}[r]
        0 & -1 & 1 \\
        -1 & 0 & -1 \\
        0 & -1 & 0
      \end{psmallmatrix*}
    & \parbox[m][4.5pc][c]{1pc}{$0$} \\ \hline
    \begin{psmallmatrix*}[r]
      1 & 0 & 0 & 0 & 0 & 0 \\
      0 & 1 & 0 & 0 & 0 & 0 \\
      0 & 0 & 1 & 0 & 0 & 0 \\
      0 & 0 & 0 & 1 & 0 & 0 \\
      0 & 0 & 0 & 0 & 1 & 0 \\
      0 & 0 & 0 & 0 & 0 & 1
    \end{psmallmatrix*}
    &
      \begin{psmallmatrix*}[r]
        1 & 0 & 0 & 0 & 0 & 0 \\
        0 & 1 & 0 & 0 & 0 & 0 \\
        0 & 0 & 1 & 0 & 0 & 0 \\
        0 & 0 & 0 & 1 & 0 & 0 \\
        0 & 0 & 0 & 0 & 1 & 0\\
        0 & 0 & 0 & 0 & 0 & 1
      \end{psmallmatrix*}
    &
      \begin{psmallmatrix*}[r]
        0 & -1 & 0 \\
        -1 & 0 & -1 \\
        0 & -1 & 0
      \end{psmallmatrix*}
    & \parbox[m][3.7pc][c]{5pc}{$\langle e_2 - e_4 + e_6 \rangle$} \\ \hline
    \begin{psmallmatrix*}[r]
      1 & 0 & 0 & 0 & 0 & 0 \\
      0 & 1 & 0 & 0 & 0 & 0 \\
      0 & 0 & 1 & 0 & 0 & 0 \\
      0 & 0 & 0 & 1 & 0 & 0 \\
      0 & 0 & 2 & 1 & 2 & 0 \\
      0 & 0 & 0 & \frac{5}{4} & \frac{1}{2} & 1
    \end{psmallmatrix*}
    &
      \begin{psmallmatrix*}[r]
        1 & 0 & 0 & 0 & 0 & 0 \\
        0 & 1 & 0 & 0 & 0 & 0 \\
        0 & 0 & 5 & 2 & 4 & 0 \\
        0 & 0 & 2 & \frac{57}{16} & \frac{21}{8} & \frac{5}{4} \\
        0 & 0 & 4 & \frac{21}{8} & \frac{17}{4} & \frac{1}{2} \\
        0 & 0 & 0 & \frac{5}{4} & \frac{1}{2} & 1
      \end{psmallmatrix*}
    &
      \begin{psmallmatrix*}[r]
        0 & 0 & 0 \\
        0 & 1 & 0 \\
        0 & 0 & 0  
      \end{psmallmatrix*}
    & \parbox[m][5.1pc][c]{9pc}{$\langle e_2 - 4 e_6, e_4 - \frac{1}{2} e_5 - 5 e_6\rangle$} \\ \hline
    \begin{psmallmatrix*}[r]
      1 & 0 & 0 & 0 & 0 & 0 \\
      0 & 1 & 0 & 0 & 0 & 0 \\
      0 & 0 & 1 & 0 & 0 & 0 \\
      0 & 0 & 0 & 1 & 0 & 0 \\
      0 & 0 & 1 & 0 & 1 & 0 \\
      0 & 0 & 0 & 2 & 0 & 1
    \end{psmallmatrix*}
    &
      \begin{psmallmatrix*}[r]
        1 & 0 & 0 & 0 & 0 & 0 \\
        0 & 1 & 0 & 0 & 0 & 0 \\
        0 & 0 & 2 & 0 & 1 & 0 \\
        0 & 0 & 0 & 5 & 0 & 2 \\
        0 & 0 & 1 & 0 & 1 & 0 \\
        0 & 0 & 0 & 2 & 0 & 1
      \end{psmallmatrix*}
    &
      \begin{psmallmatrix*}[r]
        0 & 0 & 0 \\
        0 & 0 & 0 \\
        0 & 0 & 0
      \end{psmallmatrix*}
    &
      \parbox[m][3.6pc][c]{9pc}{$\langle e_2 - e_6, e_3 - 2 e_5, e_4 - 3 e_6 \rangle$}
    \\
    \hline
  \end{array}
  \)
  \label{tab:h28-examples-index}
\end{table}

\begin{proposition}
  \it For any left-invariant metric on $H_{28}$ we have that
  \begin{equation*}
    Z(\mathfrak h_{28}) \cap \mf s = \{0\}.
  \end{equation*}
\end{proposition}

\begin{proof}
  Notice that $Z(\mathfrak h_{28}) = \mathbb R e_6$ and assume that
  $e_6 \in \mathfrak s$ where $\mathfrak s$ is the distribution of
  symmetry of some $g_\sigma$, $\sigma \in \Sigma$. Then from
  \eqref{eq:6} we have that
  \begin{equation*}
    0 = g_\sigma([e_1, e_6], e_5) + g_\sigma([e_1, e_5], e_6) +
    g_\sigma([e_6, e_5], e_1) = -s_9^2
  \end{equation*}
  which is a contradiction since $g_\sigma$ is non-degenerate. 
\end{proof}

\section{Application to nilsoliton metrics on CSLAT}
\label{sec:application-to-nilsolitons}

Let $H$ be a nilpotent simply connected Lie group endowed with a left-invariant metric $g$. Let us denote by $\mathfrak h$ the Lie algebra of $H$ and with the same symbol $g$ the inner product on $\mf h$ induced by the metric. We say that $g$ is a \emph{nilsoliton} metric if
\begin{equation} \label{eq:nilsoliton}
  \Ric_g = c  \id_{\mf h} + D
\end{equation}
for some $c \in \mathbb R$ and $D \in \Der(\mf h)$, where $\Ric_g$ is the Ricci operator of $g$ (at the identity element of $H$). Recall that a nilsoliton metric on $H$ is unique up to scaling. The classification of nilsoliton metrics is known in low dimensions. In particular the classification in dimension $6$ is obtained in  \cite{willRankoneEinsteinSolvmanifolds2003} (see also \cite{willSpaceSolvsolitonsLow2011}). Since nilsoliton metrics are Ricci soliton metrics, they have distinguished geometric properties, so it is expected that these metrics possess non-trivial index of symmetry. In the next result we prove that this is indeed the case for CSLATs of dimension $6$. 

\begin{theorem}\label{sec:application-to-nilsolitons-1}
  \it Let $\mf h$ be a CSLAT of dimension $6$ and let $H$ be the simply connected Lie group with Lie algebra $\mf h$. Assume that $H$ admits a left-invariant metric with nontrivial index of symmetry. If $g$ is a nilsoliton metric on $H$, then $i_{\mf s}(H, g) > 0$.
\end{theorem} 

\begin{proof}
  From Theorem \ref{sec:index-symmetry-csla} we know that $\mf h \in \{\mf h_9, \mf h_{10}, \mf h_{21}, \mf h_{22}, \mf h_{28}\}$. While the classification of 6-dimensional nilsolitons exists, finding an explicit isometry between a nilsoliton in this classifications and a metric from Section \ref{sec:moduli-space-left-1} might be difficult. In fact, it is easier to solve equation \ref{eq:nilsoliton} directly in $\Sigma \simeq \mathcal{M}(\mathfrak h) / {\sim_0}$. In order to compute the Ricci operator we will use the well-known formula 
  \begin{equation*}
    g(\Ric_g \tilde e_j, \tilde e_h) = \frac12\sum_{i,k} c_{iki} (c_{kjh} + c_{khj}) + \frac12 c_{ikh} c_{ikj} - c_{ijk} c_{khi} + c_{iki} c_{jhk} - c_{ijk} c_{ihk} 
  \end{equation*}
  where $\tilde e_1, \ldots, \tilde e_6$ is an orthonormal basis and
  \begin{equation*}
    c_{ijk} = g([\tilde e_i, \tilde e_j], \tilde e_k ).
  \end{equation*} 
  Assume that $g = g_\sigma$ for some $\sigma \in \Sigma$ and that $\tilde e_1, \ldots, \tilde e_6$ are obtained from the standard basis $e_1, \ldots, e_6$ via the Gram-Schmidt process. Since the Ricci operator is symmetric and any derivation has triangular form in the standard basis, the equation $\Ric_g = c \id_{\mf h} + D$, for $D \in \Der(\mf h)$ implies that $\Ric_g$ and $D$ are diagonal in the standard basis. So we can assume further that $\sigma$, and hence $g_\sigma$ are diagonal matrices (we will verify this fact a posteriori by solving the nilsoliton equation with $\sigma$ diagonal, since the nilsoliton metric is unique up to scaling \cite{lauretRicciSolitonSolvmanifolds2011}). In this case one has
  \begin{equation*}
    g(\Ric_g \tilde e_j, \tilde e_h) = g(\Ric_g e_j, e_h).
  \end{equation*}

  With these simplifications it is not hard to solve equation \ref{eq:nilsoliton} (see the attached notebook \cite[Notebook 10]{cardosoAuxiliarySageMathNotebooks2024}) and we get that the nilsoliton metrics $g_i$ on $\mf h_i$ are given by
  \begin{align*}
    g_9 & = \diag(1, 1, 2r, 1, r, r^2), \\ 
    g_{10} & = \diag(1, 1, 1, r, \tfrac{1}{2} \, r, r^{2}), \\ 
    g_{21} & = \diag(1, 1, 1, \tfrac{1}{2} \, 3^{\frac{1}{3}} 2^{\frac{1}{3}} r^{\frac{2}{3}}, \tfrac{1}{3} \, 3^{\frac{2}{3}} 2^{\frac{2}{3}} r^{\frac{4}{3}}, r^{2}), \\ 
    g_{22} & = \diag(1, \tfrac{1}{2} \, 3^{\frac{1}{4}} 2^{\frac{3}{4}} \sqrt{r}, 1, \tfrac{1}{2} \, \sqrt{3} \sqrt{2} r, \tfrac{1}{2} \, 3^{\frac{3}{4}} 2^{\frac{1}{4}} r^{\frac{3}{2}}, r^{2}), \\ 
    g_{28} & = \diag(1, 1, \tfrac{1}{3} \, \sqrt{3} \sqrt{2} \sqrt{r}, r, \tfrac{1}{2} \, \sqrt{3} \sqrt{2} r^{\frac{3}{2}}, r^{2}),  
  \end{align*}
  for $r > 0$. Now from Theorems \ref{sec:index-symmetry-csla-1}, \ref{sec:index-symmetry-csla-2}, \ref{sec:index-symmetry-csla-3}, \ref{sec:index-symmetry-csla-4} and \ref{sec:index-symmetry-csla-5} we get that the index of symmetry of these metrics is nontrivial.
\end{proof}

\begin{remark}
  Every nilsoliton metric in Theorem \ref{sec:application-to-nilsolitons-1} has $i_{\mf s}(H, g) = 1$. So it is quite surprising that the index of symmetry of a nilsoliton metric is not always maximal (since $H_{22}$ and $H_{28}$ admit metrics with index of symmetry $3$). Notice however that the distribution of symmetry is well behaved with respect to the central distribution. Namely $\mf s \subset Z(\mf h)$ when this property is not obstructed by the underlying Lie algebra structure (notice that no metric on $H_{10}$ or $H_{28}$ has well-behaved distribution of symmetry with respect to the central distribution).
\end{remark}

\begin{remark}
  For the sake of completeness, we will now identify the (normalized) nilsoliton metrics of Will's classification with their counterparts in the proof of Theorem \ref{sec:application-to-nilsolitons-1}. Following the notation from \cite{willSpaceSolvsolitonsLow2011}, the $6$-dimensional nilsoliton metrics with nontrivial index of symmetry are the ones given by 
  \begin{align}
    \mf n_9 & = (0, 0, 0, 0, 2^\frac12 \, \mathbf{12}, \mathbf{14} + 2^\frac12 \, \mathbf{25}), \notag \\
    \mf n_{10} & = (0, 0, 0, 2^\frac12 \mathbf{12}, \mathbf{13}, 2^\frac12 \mathbf{14}), \notag \\
    \mf n_{21} & = (0, 0, 0, 3^\frac12 \, \mathbf{12}, 2 \, \mathbf{14}, 3^\frac12 \, \mathbf{15}), \label{eq:nilsolitons-will} \\
    \mf n_{22} & = (0, 0, 0, 3^\frac12 \mathbf{12}, 3^\frac12 \mathbf{14}, 2^\frac12\mathbf{15} + 2^\frac12 \mathbf{24}), \notag \\
    \mf n_{28} & = (0, 0, 2 \, \mathbf{12}, 6^\frac12 \, \mathbf{13}, 6^\frac12 \, \mathbf{14}, 2 \, \mathbf{15}). \notag 
  \end{align}
  We briefly recall this notation since it slightly differs from ours. Let us consider, for example, the case of $\mf n_{22}$. The above notation means that there exist an orthonormal basis $x_1, \ldots, x_6$ of $\mf n_{22}$ with nontrivial structure coefficients given by  
  \begin{align*}
    [x_1, x_2] = \sqrt 3\, x_4 && 
    [x_1, x_4] = \sqrt 3 \, x_5 &&
    [x_1, x_5] = [x_2, x_4] = \sqrt 2 \, x_6.
  \end{align*}
  To continue with the example, if we want to find $r > 0$ such that $\mf n_{22}$ is isometric to $(\mf h_{22}, g_{22})$, we can define the Lie algebra isomorphism $\varphi$ that maps  
  \begin{align*}
    \varphi(x_1) &= e_1 &&
    \varphi(x_2) = -\tfrac13 \sqrt3 \, e_2 &&
    \varphi(x_3) = e_3 \\
    \varphi(x_4) &= \tfrac13 \, e_4 &&
    \varphi(x_5) = -\tfrac19 \sqrt3 \, e_5 &&
    \varphi(x_6) = \tfrac1{18} \sqrt6 \, e_6 
  \end{align*}
  and notice that $\varphi$ is an isometry if and only if $r = 3 \sqrt6$. With this very same idea we see that the nilsoliton metrics from \eqref{eq:nilsolitons-will} are isometric to   
  \begin{align*}
    g_9|_{r = 2} & = \diag(1, 1, 4, 1, 2, 4), \\ 
    g_{10}|_{r = 2} & = \diag(1, 1, 1, 2, 1, 4), \\ 
    g_{21}|_{r = 6} & = \diag(1, 1, 1, 3, 12, 36), \\ 
    g_{22}|_{r = 3 \sqrt 6} & = \diag(1, 3, 1, 9, 27, 54), \\ 
    g_{28}|_{r = 24} & = \diag(1, 1, 4, 24, 144, 576),
  \end{align*}
  respectively.
\end{remark}

\section*{Acknowledgments}

The authors would like to thank Jorge Lauret for pointing out the problem of studying the index of symmetry of nilsoliton metrics. The authors would also like to thank the anonymous referee who identified, in an earlier version of this article, the omission of some cases in Proposition~\ref{sec:char-solv-lie}. This work was supported by CONICET, UNR and partially supported by SeCyT-UNR and ANPCyT.

\appendix

\section{Tables}
\label{sec:appendix}

In this section we collect the tables with our main classificatory results.

\begin{table}[ht]
  \caption{Automorphism group of $3$-step nilpotent CSLATs}
  \centering
  \(
  \begin{array}{|l|l|l|}
    \hline
    \mf h & \Aut_0(\mf h) & D \\ \hline\hline
    \parbox[m][5pc][c]{1pc}{$\mf h_9$} &
              \begin{psmallmatrix*}[r]
                a_{0} & 0 & 0 & 0 & 0 & 0 \\
                a_{1} & a_{2} & 0 & 0 & 0 & 0 \\
                a_{3} & a_{4} & a_{0}^{2} & 0 & 0 & 0 \\
                a_{5} & a_{6} & a_{7} & a_{8} & 0 & 0 \\
                a_{9} & a_{10} & a_{0} a_{1} & 0 & a_{0} a_{2} & 0 \\
                a_{11} & a_{12} & a_{13} & a_{14} & -a_{0} a_{10} - a_{2} a_{3} + a_{1} a_{4} & a_{0}^{2} a_{2}
              \end{psmallmatrix*},
              \begin{smallmatrix*}[c]
                a_0, a_2, a_8 > 0
              \end{smallmatrix*}
    &
      \begin{smallmatrix*}[l]
        \langle \diag(-1, 1, 1, 1, -1, 1), \\
        \hspace{.3pc} \diag(1, -1, 1, 1, -1, -1), \\
        \hspace{.3pc} \diag(1, 1, 1, -1, 1, 1) \rangle
      \end{smallmatrix*}
    \\ \hline
    \parbox[m][4.4pc][c]{1.1pc}{$\mf h_{10}$} &
                 \begin{psmallmatrix*}[r]
                   a_{0} & 0 & 0 & 0 & 0 & 0 \\
                   a_{1} & a_{2} & 0 & 0 & 0 & 0 \\
                   a_{3} & a_{4} & a_{5} & 0 & 0 & 0 \\
                   a_{6} & a_{7} & a_{8} & a_{0} a_{2} & 0 & 0 \\
                   a_{9} & a_{10} & a_{11} & a_{0} a_{4} & a_{0} a_{5} & 0 \\
                   a_{12} & a_{13} & a_{14} & a_{0} a_{7} & a_{0} a_{8} & a_{0}^{2} a_{2}
                 \end{psmallmatrix*},
                 \begin{smallmatrix*}[c]
                   a_0, a_2, a_5 > 0
                 \end{smallmatrix*}
    &
      \begin{smallmatrix*}[l]
        \langle \diag(-1, 1, 1, -1, -1, 1), \\
        \hspace{.3pc} \diag(1, -1, 1, -1, 1, -1), \\
        \hspace{.3pc} \diag(1, 1, -1, 1, -1, 1) \rangle
      \end{smallmatrix*}
    \\ \hline
    \parbox[m][4.5pc][c]{1.1pc}{$\mf h_{11}$} & \parbox[m][5.1pc][c]{17pc}{
                 $\begin{psmallmatrix*}[r]
                   a_{0} & 0 & 0 & 0 & 0 & 0 \\
                   a_{1} & a_{2} & 0 & 0 & 0 & 0 \\
                   a_{3} & a_{4} & a_{0}^{2} & 0 & 0 & 0 \\
                   a_{5} & a_{6} & a_{7} & a_{0} a_{2} & 0 & 0 \\
                   a_{8} & a_{9} & a_{10} & a_{0} a_{4} & a_{0}^{3} & 0 \\
                   a_{11} & a_{12} & a_{13} & -a_{2} a_{3} + a_{1} a_{4} + a_{0} a_{6} & a_{0}^{2} a_{1} + a_{0} a_{7} & a_{0}^{2} a_{2}
                 \end{psmallmatrix*}$,
                 \parbox[m][1pc][c]{3pc}{
                  $\begin{smallmatrix*}[c]
                   a_0, a_2 > 0
                 \end{smallmatrix*}$
                 }}
    &
      \begin{smallmatrix*}[l]
        \langle \diag(-1, 1, 1, -1, -1, 1), \\
        \hspace{.3pc} \diag(1, -1, 1, -1, 1, -1) \rangle
      \end{smallmatrix*} \\ \hline
      \parbox[m][4.5pc][c]{1.1pc}{$\mf h_{12}$} &
                 \begin{psmallmatrix*}[r]
                   a_{0} & 0 & 0 & 0 & 0 & 0 \\
                   0 & a_{1} & 0 & 0 & 0 & 0 \\
                   a_{2} & a_{3} & a_{4} & 0 & 0 & 0 \\
                   a_{5} & a_{6} & 0 & a_{0} a_{1} & 0 & 0 \\
                   a_{7} & a_{8} & a_{9} & a_{0} a_{3} & a_{0} a_{4} & 0 \\
                   a_{10} & a_{11} & a_{12} & -a_{1} a_{5} & 0 & a_{0} a_{1}^{2}                    
                 \end{psmallmatrix*},
                 \begin{smallmatrix*}[c]
                   a_0, a_1, a_4 > 0
                 \end{smallmatrix*}
    &
      \begin{smallmatrix*}[l]
        \langle \diag(-1, 1, 1, -1, -1, -1), \\
        \hspace{.3pc} \diag(1, -1, 1, -1, 1, 1), \\
        \hspace{.3pc} \diag(1, 1, -1, 1, -1, 1) \rangle
      \end{smallmatrix*}
    \\ \hline
    \parbox[m][5.1pc][c]{1.1pc}{$\mf h_{13}$} &
                 \begin{psmallmatrix*}[r]
                   a_{0} & 0 & 0 & 0 & 0 & 0 \\
                   0 & a_{1} & 0 & 0 & 0 & 0 \\
                   a_{2} & a_{3} & a_{0} a_{1} & 0 & 0 & 0 \\
                   a_{4} & a_{5} & 0 & a_{0} a_{1} & 0 & 0 \\
                   a_{6} & a_{7} & a_{8} & a_{0} a_{3} + a_{0} a_{5} & a_{0}^{2} a_{1} & 0 \\
                   a_{9} & a_{10} & a_{11} & -a_{1} a_{4} & 0 & a_{0} a_{1}^{2}
                 \end{psmallmatrix*},
                 \begin{smallmatrix*}[c]
                   a_0, a_1 > 0
                 \end{smallmatrix*}
    &
      \begin{smallmatrix*}[l]
        \langle \diag(-1, 1, -1, -1, 1, -1), \\
        \hspace{.3pc} \diag(1, -1, -1, -1, -1, 1), \\
        \begin{psmallmatrix*}[r]
          0 & -1 & 0 & 0 & 0 & 0 \\
          1 & 0 & 0 & 0 & 0 & 0 \\
          0 & 0 & -1 & 0 & 0 & 0 \\
          0 & 0 & 1 & 1 & 0 & 0 \\
          0 & 0 & 0 & 0 & 0 & -1 \\
          0 & 0 & 0 & 0 & 1 & 0
        \end{psmallmatrix*}
        \rangle
      \end{smallmatrix*}
    \\ \hline
    \parbox[m][5pc][c]{1.1pc}{$\mf h_{14}$} &
                 \begin{psmallmatrix*}[r]
                   a_{0} & 0 & 0 & 0 & 0 & 0 \\
                   a_{1} & a_{2} & 0 & 0 & 0 & 0 \\
                   a_{3} & a_{4} & a_{2}^{2} & 0 & 0 & 0 \\
                   a_{5} & a_{6} & 0 & a_{0} a_{2} & 0 & 0 \\
                   a_{7} & a_{8} & a_{9} & a_{0} a_{6} & a_{0}^{2} a_{2} & 0 \\
                   a_{10} & a_{11} & a_{12} & a_{0} a_{4} + a_{2} a_{5} - a_{1} a_{6} & -a_{0} a_{1} a_{2} & a_{0} a_{2}^{2}
                 \end{psmallmatrix*},
                 \begin{smallmatrix*}[c]
                   a_0, a_2 > 0 
                 \end{smallmatrix*}
    &
      \begin{smallmatrix*}[l]
        \langle \diag(-1, 1, 1, -1, 1, -1), \\
        \hspace{.3pc} \diag(1, -1, 1, -1, -1, 1) \rangle
      \end{smallmatrix*}
    \\ \hline
    \parbox[m][5.8pc][c]{1.1pc}{$\mf h_{18}$} &
                 \begin{psmallmatrix*}[r]
                   a_{0} & 0 & 0 & 0 & 0 & 0 \\
                   a_{1} & a_{2} & 0 & 0 & 0 & 0 \\
                   a_{3} & -\frac{a_{1} a_{2}}{a_{0}} & \frac{a_{2}^{2}}{a_{0}} & 0 & 0 & 0 \\
                   a_{4} & a_{5} & 0 & a_{0} a_{2} & 0 & 0 \\
                   a_{6} & a_{7} & a_{8} & -a_{1} a_{2} & a_{2}^{2} & 0 \\
                   a_{9} & a_{10} & a_{11} & -a_{2} a_{4} + a_{1} a_{5} + a_{0} a_{7} & a_{0} a_{8} & a_{0} a_{2}^{2}
                 \end{psmallmatrix*},
                 \begin{smallmatrix*}[c]
                   a_0, a_2 > 0
                 \end{smallmatrix*}
    &
      \begin{smallmatrix*}[l]
        \langle \diag(-1, 1, -1, -1, 1, -1), \\
        \hspace{.3pc} \diag(1, -1, 1, -1, 1, 1) \rangle
      \end{smallmatrix*} \\ \hline
      \parbox[m][4.9pc][c]{1.1pc}{$\mf h_{19}^+$} &
                   \begin{psmallmatrix*}[r]
                     a_{0} & 0 & 0 & 0 & 0 & 0 \\
                     0 & a_{1} & 0 & 0 & 0 & 0 \\
                     0 & 0 & a_{2} & 0 & 0 & 0 \\
                     a_{3} & a_{4} & a_{5} & a_{1} a_{2} & 0 & 0 \\
                     \frac{a_{0} a_{5}}{a_{2}} & a_{6} & a_{7} & 0 & a_{0} a_{1} & 0 \\
                     a_{8} & a_{9} & a_{10} & -a_{2} a_{6} & -a_{0} a_{4} & a_{0} a_{1} a_{2}                      
                   \end{psmallmatrix*},
                   \begin{smallmatrix*}[c]
                     a_0, a_1, a_2 > 0 
                   \end{smallmatrix*}
    &
      \begin{smallmatrix*}[l]
        \langle \diag(-1, 1, 1, 1, -1, -1), \\
        \hspace{.3pc} \diag(1, -1, 1, -1, -1, -1) \\
        \begin{psmallmatrix}
          0 & 0 & 1 & 0 & 0 & 0 \\
          0 & 1 & 0 & 0 & 0 & 0 \\
          1 & 0 & 0 & 0 & 0 & 0 \\
          0 & 0 & 0 & 0 & 1 & 0 \\
          0 & 0 & 0 & 1 & 0 & 0 \\
          0 & 0 & 0 & 0 & 0 & 1
        \end{psmallmatrix}
        \rangle
      \end{smallmatrix*}
    \\ \hline
  \end{array} 
  \)
  \label{tab:aut-group-3-step}
\end{table}

\begin{table}[ht]
  \caption{Automorphism group of $4$-step nilpotent CSLATs}
  \centering
  \(
  \begin{array}{|l|l|l|}
    \hline
    \mf h & \Aut_0(\mf h) & D \\ \hline\hline
    \parbox[m][5pc][c]{1.1pc}{$\mf h_{21}$} &
                 \begin{psmallmatrix*}[r]
                   a_{0} & 0 & 0 & 0 & 0 & 0 \\
                   a_{1} & a_{2} & 0 & 0 & 0 & 0 \\
                   a_{3} & a_{4} & a_{5} & 0 & 0 & 0 \\
                   a_{6} & a_{7} & 0 & a_{0} a_{2} & 0 & 0 \\
                   a_{8} & a_{9} & 0 & a_{0} a_{7} & a_{0}^{2} a_{2} & 0 \\
                   a_{10} & a_{11} & a_{12} & a_{0} a_{9} & a_{0}^{2} a_{7} & a_{0}^{3} a_{2}
                 \end{psmallmatrix*},
                 \begin{smallmatrix*}[c]
                   a_1, a_2, a_5 > 0
                 \end{smallmatrix*}
    &
      \begin{smallmatrix*}[l]
        \langle \diag(-1, 1, 1, -1, 1, -1), \\
        \hspace{.3pc} \diag(1, -1, 1, -1,-1, -1), \\
        \hspace{.3pc} \diag(1, 1, -1, 1, 1, 1) \rangle
      \end{smallmatrix*}
    \\ \hline
    \parbox[m][5pc][c]{1.1pc}{$\mf h_{22}$} &
                 \begin{psmallmatrix*}[r]
                   a_{0} & 0 & 0 & 0 & 0 & 0 \\
                   a_{1} & a_{0}^{2} & 0 & 0 & 0 & 0 \\
                   a_{2} & a_{3} & a_{4} & 0 & 0 & 0 \\
                   a_{5} & a_{6} & 0 & a_{0}^{3} & 0 & 0 \\
                   a_{7} & a_{8} & 0 & a_{0} a_{6} & a_{0}^{4} & 0 \\
                   a_{9} & a_{10} & a_{11} & -a_{0}^{2} a_{5} + a_{1} a_{6} + a_{0} a_{8} & a_{0}^{3} a_{1} + a_{0}^{2} a_{6} &  a_{0}^{5}                   
                 \end{psmallmatrix*},
                 \begin{smallmatrix*}[c]
                   a_0, a_4 > 0
                 \end{smallmatrix*}
    &
      \begin{smallmatrix*}[l]
        \langle \diag(-1, 1, 1, -1, 1, -1), \\
        \hspace{.3pc} \diag(1, 1, -1, 1, 1, 1) \rangle
      \end{smallmatrix*}
    \\ \hline
    \parbox[m][5pc][c]{1.1pc}{$\mf h_{23}$} &
                 \begin{psmallmatrix*}[r]
                   a_{0} & 0 & 0 & 0 & 0 & 0 \\
                   a_{1} & a_{2} & 0 & 0 & 0 & 0 \\
                   a_{3} & a_{4} & a_{0} a_{2} & 0 & 0 & 0 \\
                   a_{5} & a_{6} & a_{0} a_{4} & a_{0}^{2} a_{2} & 0 & 0 \\
                   a_{7} & a_{8} & -a_{2} a_{3} + a_{1} a_{4} & a_{0} a_{1} a_{2} & a_{0} a_{2}^{2} & 0 \\
                   a_{9} & a_{10} & a_{0} a_{6} & a_{0}^{2} a_{4} & 0 & a_{0}^{3} a_{2}
                 \end{psmallmatrix*},
                 \begin{smallmatrix*}[c]
                   a_0, a_2 > 0
                 \end{smallmatrix*}
    &
      \begin{smallmatrix*}[l]
        \langle \diag(-1, 1, -1, 1, -1, -1), \\
        \hspace{.3pc} \diag(1, -1, -1, -1, 1, -1) \rangle
      \end{smallmatrix*} \\ \hline
    \mf h_{24} & \parbox[m][5.3pc][c]{20pc}{$
                 \begin{psmallmatrix*}[r]
                   a_{0} & 0 & 0 & 0 & 0 & 0 \\
                   a_{1} & a_{0}^{2} & 0 & 0 & 0 & 0 \\
                   a_{2} & a_{3} & a_{0}^{3} & 0 & 0 & 0 \\
                   a_{4} & a_{5} & 0 & a_{0}^{3} & 0 & 0 \\
                   a_{6} & a_{7} & -a_{0}^{2} a_{1} & a_{0} a_{5} & a_{0}^{4} & 0 \\
                   a_{8} & a_{9} & a_{10} & -a_{0}^{2} a_{2} - a_{0}^{2} a_{4} + a_{1} a_{3} + a_{1} a_{5} + a_{0} a_{7} & a_{0}^{3} a_{1} + a_{0}^{2} a_{5} & a_{0}^{5}
                 \end{psmallmatrix*}$, 
                 $\begin{smallmatrix*}[c]
                   a_0 > 0
                 \end{smallmatrix*}$}
    &
      \begin{smallmatrix*}[l]
        \langle \diag(-1, 1, -1, -1, 1, -1) \rangle
      \end{smallmatrix*}
    \\ \hline
    \parbox[m][5pc][c]{1.1pc}{$\mf h_{25}$} &
                 \begin{psmallmatrix*}[r]
                   a_{0} & 0 & 0 & 0 & 0 & 0 \\
                   a_{1} & a_{2} & 0 & 0 & 0 & 0 \\
                   a_{3} & a_{4} & a_{0}^{3} & 0 & 0 & 0 \\
                   a_{5} & a_{6} & 0 & a_{0} a_{2} & 0 & 0 \\
                   a_{7} & a_{8} & -a_{0}^{2} a_{1} & a_{0} a_{6} & a_{0}^{2} a_{2} & 0 \\
                   a_{9} & a_{10} & a_{11} & -a_{2} a_{3} + a_{1} a_{4} + a_{0} a_{8} & a_{0}^{2} a_{6} & a_{0}^{3} a_{2}
                 \end{psmallmatrix*},
                 \begin{smallmatrix*}[c]
                   a_0, a_2 > 0
                 \end{smallmatrix*}
    &
      \begin{smallmatrix*}[l]
        \langle \diag(-1, 1, -1, -1, 1, -1), \\
        \hspace{.3pc} \diag(1, -1, 1, -1, -1, -1) \rangle
      \end{smallmatrix*}
    \\ \hline
    \parbox[m][5pc][c]{1.1pc}{$\mf h_{26}^-$} &
                   \begin{psmallmatrix*}[r]
                     a_{0} & 0 & 0 & 0 & 0 & 0 \\
                     0 & a_{1} & 0 & 0 & 0 & 0 \\
                     a_{2} & a_{3} & a_{0} a_{1} & 0 & 0 & 0 \\
                     a_{4} & a_{5} & -a_{0} a_{3} & a_{0}^{2} a_{1} & 0 & 0 \\
                     a_{6} & a_{7} & a_{1} a_{2} & 0 & a_{0} a_{1}^{2} & 0 \\
                     a_{8} & a_{9} & -a_{1} a_{4} + a_{0} a_{7} & -a_{0} a_{1} a_{2} & a_{0} a_{1} a_{3} & a_{0}^{2} a_{1}^{2}
                   \end{psmallmatrix*},
                   \begin{smallmatrix*}[c]
                     a_0, a_1 > 0
                   \end{smallmatrix*}
    &
      \begin{smallmatrix*}[l]
        \langle \diag(-1, 1, -1, 1, -1, 1), \\
        \begin{psmallmatrix*}[r]
          0 & 1 & 0 & 0 & 0 & 0 \\
          1 & 0 & 0 & 0 & 0 & 0 \\
          0 & 0 & -1 & 0 & 0 & 0 \\
          0 & 0 & 0 & 0 & -1 & 0 \\
          0 & 0 & 0 & -1 & 0 & 0 \\
          0 & 0 & 0 & 0 & 0 & -1
        \end{psmallmatrix*}
        \rangle
      \end{smallmatrix*}
    \\ \hline
    \parbox[m][5pc][c]{1.1pc}{$\mf h_{27}$} &
                 \begin{psmallmatrix*}[r]
                   a_{0} & 0 & 0 & 0 & 0 & 0 \\
                   a_{1} & a_{2} & 0 & 0 & 0 & 0 \\
                   0 & 0 & a_{0}^{2} & 0 & 0 & 0 \\
                   a_{3} & a_{4} & a_{0} a_{1} & a_{0} a_{2} & 0 & 0 \\
                   a_{5} & a_{6} & a_{0} a_{3} & a_{0} a_{4} & a_{0}^{2} a_{2} & 0 \\
                   a_{7} & a_{8} & a_{9} & a_{0} a_{6} & a_{0}^{2} a_{4} & a_{0}^{3} a_{2}
                 \end{psmallmatrix*},
                 \begin{smallmatrix*}[c]
                   a_0, a_2 > 0
                 \end{smallmatrix*}
    &
      \begin{smallmatrix*}[l]
        \langle \diag(-1, 1, 1, -1, 1, -1), \\
        \hspace{.3pc} \diag(1, -1, 1, -1, -1, -1) \rangle
      \end{smallmatrix*}
      \\ \hline
  \end{array}  
  \)
  \label{tab:aut-group-4-step}
\end{table}

\begin{table}[ht]
  \caption{Automorphism group of $5$-step nilpotent CSLATs}
  \centering
  \(
  \begin{array}{|l|l|l|}
    \hline
    \mf h & \Aut_0(\mf h) & D \\ \hline\hline
    \parbox[m][5pc][c]{1.1pc}{$\mf h_{28}$} &
                 \begin{psmallmatrix*}[r]
                   a_{0} & 0 & 0 & 0 & 0 & 0 \\
                   a_{1} & a_{2} & 0 & 0 & 0 & 0 \\
                   a_{3} & a_{4} & a_{0} a_{2} & 0 & 0 & 0 \\
                   a_{5} & a_{6} & a_{0} a_{4} & a_{0}^{2} a_{2} & 0 & 0 \\
                   a_{7} & a_{8} & a_{0} a_{6} & a_{0}^{2} a_{4} & a_{0}^{3} a_{2} & 0 \\
                   a_{9} & a_{10} & a_{0} a_{8} & a_{0}^{2} a_{6} & a_{0}^{3} a_{4} & a_{0}^{4} a_{2}
                 \end{psmallmatrix*},
                 \begin{smallmatrix*}[c]
                   a_0, a_2 > 0
                 \end{smallmatrix*}
    &
      \begin{smallmatrix*}[l]
        \langle \diag(-1, 1, -1, 1, -1, 1), \\
        \hspace{.3pc} \diag(1, -1, -1, -1, -1, -1) \rangle
      \end{smallmatrix*}
    \\ \hline
    \parbox[m][5.2pc][c]{1.1pc}{$\mf h_{29}$} &
                 \begin{psmallmatrix*}[r]
                   a_{0} & 0 & 0 & 0 & 0 & 0 \\
                   a_{1} & a_{0}^{3} & 0 & 0 & 0 & 0 \\
                   a_{2} & a_{3} & a_{0}^{4} & 0 & 0 & 0 \\
                   a_{4} & a_{5} & a_{0} a_{3} & a_{0}^{5} & 0 & 0 \\
                   a_{6} & a_{7} & a_{0} a_{5} & a_{0}^{2} a_{3} & a_{0}^{6} & 0 \\
                   a_{8} & a_{9} & -a_{0}^{3} a_{2} + a_{1} a_{3} + a_{0} a_{7} & a_{0}^{4} a_{1} + a_{0}^{2} a_{5} & a_{0}^{3} a_{3} & a_{0}^{7}
                 \end{psmallmatrix*},
                 \begin{smallmatrix*}[c]
                   a_0 > 0
                 \end{smallmatrix*}
    &
      \begin{smallmatrix*}[l]
           \langle \diag(-1, -1, 1, -1, 1, -1) \rangle
      \end{smallmatrix*}
    \\ \hline
    \parbox[m][5.2pc][c]{1.1pc}{$\mf h_{30}$} &
                 \begin{psmallmatrix*}[r]
                   a_{0} & 0 & 0 & 0 & 0 & 0 \\
                   0 & a_{0}^{2} & 0 & 0 & 0 & 0 \\
                   a_{1} & a_{2} & a_{0}^{3} & 0 & 0 & 0 \\
                   a_{3} & a_{4} & a_{0} a_{2} & a_{0}^{4} & 0 & 0 \\
                   a_{5} & a_{6} & -a_{0}^{2} a_{1} + a_{0} a_{4} & a_{0}^{2} a_{2} & a_{0}^{5} & 0 \\
                   a_{7} & a_{8} & -a_{0}^{2} a_{3} + a_{0} a_{6} & -a_{0}^{3} a_{1} + a_{0}^{2} a_{4} & a_{0}^{3} a_{2} & a_{0}^{6}
                 \end{psmallmatrix*},
                 \begin{smallmatrix*}[c]
                   a_0 > 0
                 \end{smallmatrix*}
    &
      \begin{smallmatrix*}[l]
           \langle \diag(-1, 1, -1, 1, -1, 1) \rangle
      \end{smallmatrix*} \\ \hline
    \mf h_{31} & \parbox[m][6.7pc][c]{20pc}{
                 $\begin{psmallmatrix*}[r]
                   a_{0} & 0 & 0 & 0 & 0 & 0 \\
                   0 & a_{1} & 0 & 0 & 0 & 0 \\
                   a_{2} & a_{3} & a_{0} a_{1} & 0 & 0 & 0 \\
                   a_{4} & \frac{a_{3}^{2}}{2 \, a_{1}} & a_{0} a_{3} & a_{0}^{2} a_{1} & 0 & 0 \\
                   a_{5} & a_{6} & \frac{a_{0} a_{3}^{2}}{2 \, a_{1}} & a_{0}^{2} a_{3} & a_{0}^{3} a_{1} & 0 \\
                   a_{7} & a_{8} & \frac{a_{2} a_{3}^{2} - 2 \, a_{1} a_{3} a_{4} + 2 \, a_{1}^{2} a_{5}}{2 \, a_{1}} & a_{0} a_{2} a_{3} - a_{0} a_{1} a_{4} & a_{0}^{2} a_{1} a_{2} & a_{0}^{3} a_{1}^{2}
                 \end{psmallmatrix*}$,
                 $\begin{smallmatrix*}[c]
                   a_0, a_1 > 0
                 \end{smallmatrix*}$}
    &
      \begin{smallmatrix*}[l]
        \langle \diag(-1, 1, -1, 1, -1, -1), \\
        \hspace{.3pc} \diag(1, -1, -1, -1, -1, 1) \rangle
      \end{smallmatrix*}
    \\ \hline
    \mf h_{32} & \parbox[m][7.4pc][c]{18pc}{
                 $\begin{psmallmatrix*}[r]
                   a_{0} & 0 & 0 & 0 & 0 & 0 \\
                   0 & a_{0}^{2} & 0 & 0 & 0 & 0 \\
                   a_{1} & a_{2} & a_{0}^{3} & 0 & 0 & 0 \\
                   a_{3} & \frac{a_{2}^{2}}{2 \, a_{0}^{2}} & a_{0} a_{2} & a_{0}^{4} & 0 & 0 \\
                   a_{4} & a_{5} & -\frac{2 \, a_{0}^{3} a_{1} - a_{2}^{2}}{2 \, a_{0}} & a_{0}^{2} a_{2} & a_{0}^{5} & 0 \\
                   a_{6} & a_{7} & \frac{2 \, a_{0}^{4} a_{4} - 2 \, a_{0}^{2} a_{2} a_{3} + a_{1} a_{2}^{2}}{2 \, a_{0}^{2}} & -a_{0}^{3} a_{3} + a_{0} a_{1} a_{2} & a_{0}^{4} a_{1} & a_{0}^{7}
                 \end{psmallmatrix*}$,
                 $\begin{smallmatrix*}[c]
                   a_0 > 0
                 \end{smallmatrix*}$}
    &
      \begin{smallmatrix*}[l]
        \langle \diag(-1, 1, -1, 1, -1, -1) \rangle
      \end{smallmatrix*}
    \\ \hline
  \end{array} 
  \) 
  \label{tab:aut-group-5-step}
\end{table}

\begin{table}[ht]
  \centering
  \caption{Moduli space of left-invariant metrics and isometric
    automorphisms for $3$-step CSLATs except $\mf h_{13}$ and $\mf h_{19}^+$}
  \begin{tabular}[t]{|c|c|c|c|c|c|c|}
    \hline
    \parbox[m][1.2pc][c]{1.1pc}{$\mf h$} & $\Sigma$ & $\#\operatorname{nd}(\Sigma)=0$ & $\{e\}$ & $\mathbb{Z}_2$ & $\mathbb{Z}_2^2$ & $\mathbb{Z}_2^3$ \\
    \hline\hline
    \multirow{4}{*}{$\mf h_9$} & \multirow{4}{*}{\(
                                 \begin{psmallmatrix}
                                   1 & 0 & 0 & 0 & 0 & 0 \\
                                   0 & 1 & 0 & 0 & 0 & 0 \\
                                   0 & 0 & s_{0} & 0 & 0 & 0 \\
                                   0 & 0 & 0 & 1 & 0 & 0 \\
                                   0 & 0 & s_{1} & s_{2} & s_{3} & 0 \\
                                   0 & 0 & 0 & 0 & s_{4} & s_{5}
                                 \end{psmallmatrix}
                                 \)} & 0 & 1 &   &   & \\ \cline{3-7}
            &  & 1 &   & 3 &   & \\ \cline{3-7}
              &     & 2 &   &   & 3 & \\ \cline{3-7}
              &     & 3 &   &   &   & 1 \\
    \hline
    \multirow{4}{*}{$\mf h_{10}$} & \multirow{4}{*}{\(
                                    \begin{psmallmatrix}
                                      1 & 0 & 0 & 0 & 0 & 0 \\
                                      0 & 1 & 0 & 0 & 0 & 0 \\
                                      0 & 0 & 1 & 0 & 0 & 0 \\
                                      0 & 0 & 0 & s_{0} & 0 & 0 \\
                                      0 & 0 & 0 & s_{1} & s_{2} & 0 \\
                                      0 & 0 & 0 & s_{3} & s_{4} & s_{5}
                                    \end{psmallmatrix}
                                    \)} & 0 &  & 1 &   & \\ \cline{3-7}
            & & 1 & & 3 &   & \\ \cline{3-7}
            & & 2 & &   & 3 & \\ \cline{3-7}
            & & 3 & &   &   & 1 \\
    \hline
    \multirow{4}{*}{$\mf h_{11}$} & \multirow{4}{*}{\(
                                    \begin{psmallmatrix}
                                      1 & 0 & 0 & 0 & 0 & 0 \\
                                      0 & 1 & 0 & 0 & 0 & 0 \\
                                      0 & 0 & s_{0} & 0 & 0 & 0 \\
                                      0 & 0 & 0 & s_{1} & 0 & 0 \\
                                      0 & 0 & 0 & s_{2} & s_{3} & 0 \\
                                      0 & 0 & 0 & s_{4} & s_{5} & s_{6}
                                    \end{psmallmatrix}
                                    \)} & 0 & 1 & & & - \\ \cline{3-7}
            & & 1 & 3 & & & - \\ \cline{3-7}
            & & 2 & & 3 & & - \\ \cline{3-7}
            & & 3 & & & 1 & - \\ 
    \hline
    \multirow{6}{*}{$\mf h_{12}$} & \multirow{6}{*}{\(
                                    \begin{psmallmatrix}
                                      1 & 0 & 0 & 0 & 0 & 0 \\
                                      s_{0} & 1 & 0 & 0 & 0 & 0 \\
                                      0 & 0 & 1 & 0 & 0 & 0 \\
                                      0 & 0 & s_{1} & s_{2} & 0 & 0 \\
                                      0 & 0 & 0 & s_{3} & s_{4} & 0 \\
                                      0 & 0 & 0 & s_{5} & s_{6} & s_{7}
                                    \end{psmallmatrix}
                                    \)} & 0 & 1 & & & \\ \cline{3-7}
            & & 1 & 5 &    &   & \\ \cline{3-7}
            & & 2 & 8 & 2  &   & \\ \cline{3-7}
            & & 3 &   & 10 &   & \\ \cline{3-7}
            & & 4 &   &    & 5 & \\ \cline{3-7}
            & & 5 &   &    &   & 1 \\
    \hline
    \multirow{5}{*}{$\mf h_{14}$} & \multirow{5}{*}{\(
                                    \begin{psmallmatrix}
                                      1 & 0 & 0 & 0 & 0 & 0 \\
                                      0 & 1 & 0 & 0 & 0 & 0 \\
                                      0 & 0 & s_{0} & 0 & 0 & 0 \\
                                      0 & 0 & s_{1} & s_{2} & 0 & 0 \\
                                      0 & 0 & 0 & s_{3} & s_{4} & 0 \\
                                      0 & 0 & 0 & s_{5} & s_{6} & s_{7}
                                    \end{psmallmatrix}
                                    \)} & 0 & 1 & & & - \\ \cline{3-7}
            & & 1 & 4 & & & - \\ \cline{3-7}
            & & 2 & 5 & 1 &   & - \\ \cline{3-7}
            & & 3 &   & 4 &   & - \\ \cline{3-7}
            & & 4 &   &   & 1 & - \\
    \hline
    \multirow{6}{*}{$\mf h_{18}$} & \multirow{6}{*}{\(
                                    \begin{psmallmatrix}
                                      1 & 0 & 0 & 0 & 0 & 0 \\
                                      0 & 1 & 0 & 0 & 0 & 0 \\
                                      0 & s_{0} & s_{1} & 0 & 0 & 0 \\
                                      0 & 0 & s_{2} & s_{3} & 0 & 0 \\
                                      0 & 0 & 0 & s_{4} & s_{5} & 0 \\
                                      0 & 0 & 0 & s_{6} & s_{7} & s_{8}
                                    \end{psmallmatrix}
                                    \)} & 0 & 1 & & & - \\ \cline{3-7}
            & & 1 & 5  &   &   & - \\ \cline{3-7}
            & & 2 & 10 &   &   & - \\ \cline{3-7}
            & & 3 & 8  & 2 &   & - \\ \cline{3-7}
            & & 4 &    & 5 &   & - \\ \cline{3-7}
            & & 5 &    &   & 1 & - \\
            \hline
  \end{tabular}
  \label{tab:sigmas-3step-no-h19}
\end{table}

\begin{table}[ht]
  \centering
  \caption{Moduli space of left-invariant metrics and isometric
    automorphisms for $4$-step CSLATs except $\mf h_{26}^-$}
  \begin{tabular}[t]{|c|c|c|c|c|c|c|}
    \hline
    \parbox[m][1.2pc][c]{1.1pc}{$\mf h$} & $\Sigma$ & $\#\operatorname{nd}(\Sigma)=0$ & $\{e\}$ & $\mathbb{Z}_2$ & $\mathbb{Z}_2^2$ & $\mathbb{Z}_2^3$ \\
    \hline\hline
    \multirow{6}{*}{$\mf h_{21}$} & \multirow{6}{*}{\(
                                    \begin{psmallmatrix}
                                      1 & 0 & 0 & 0 & 0 & 0 \\
                                      0 & 1 & 0 & 0 & 0 & 0 \\
                                      0 & 0 & 1 & 0 & 0 & 0 \\
                                      0 & 0 & s_{0} & s_{1} & 0 & 0 \\
                                      0 & 0 & s_{2} & s_{3} & s_{4} & 0 \\
                                      0 & 0 & 0 & s_{5} & s_{6} & s_{7}
                                    \end{psmallmatrix}
                                    \)} & 0 & & 1 & & \\ \cline{3-7}
            & & 1 & & 5 &   & \\ \cline{3-7}
            & & 2 & & 9 & 1 & \\ \cline{3-7}
            & & 3 & & 5 & 5 & \\ \cline{3-7}
            & & 4 & &   & 4 & 1 \\ \cline{3-7}
            & & 5 & &   &   & 1 \\
    \hline
    \multirow{6}{*}{$\mf h_{22}$} & \multirow{6}{*}{\(
                                    \begin{psmallmatrix}
                                      1 & 0 & 0 & 0 & 0 & 0 \\
                                      0 & s_{0} & 0 & 0 & 0 & 0 \\
                                      0 & 0 & 1 & 0 & 0 & 0 \\
                                      0 & 0 & s_{1} & s_{2} & 0 & 0 \\
                                      0 & 0 & s_{3} & s_{4} & s_{5} & 0 \\
                                      0 & 0 & 0 & s_{6} & s_{7} & s_{8}
                                    \end{psmallmatrix}
                                    \)} & 0 & 1 & & & - \\ \cline{3-7}
            & & 1 & 5 &   &   & - \\ \cline{3-7}
            & & 2 & 9 & 1 &   & - \\ \cline{3-7}
            & & 3 & 5 & 5 &   & - \\ \cline{3-7}
            & & 4 &   & 4 & 1 & - \\ \cline{3-7}
            & & 5 &   &   & 1 & - \\
    \hline
    \multirow{7}{*}{$\mf h_{23}$} & \multirow{7}{*}{\(
                                    \begin{psmallmatrix}
                                      1 & 0 & 0 & 0 & 0 & 0 \\
                                      0 & 1 & 0 & 0 & 0 & 0 \\
                                      0 & 0 & s_{0} & 0 & 0 & 0 \\
                                      0 & 0 & s_{1} & s_{2} & 0 & 0 \\
                                      0 & 0 & s_{3} & s_{4} & s_{5} & 0 \\
                                      0 & 0 & s_{6} & s_{7} & s_{8} & s_{9}
                                    \end{psmallmatrix}
                                    \)} & 0 & 1 & & & - \\ \cline{3-7}
            & & 1 & 6  &   &   & - \\ \cline{3-7}
            & & 2 & 15 &   &   & - \\ \cline{3-7}
            & & 3 & 18 & 2 &   & - \\ \cline{3-7}
            & & 4 & 8  & 7 &   & - \\ \cline{3-7}
            & & 5 &    & 5 & 1 & - \\ \cline{3-7}
            & & 6 &    &   & 1 & - \\
    \hline
    \multirow{6}{*}{$\mf h_{24}$} & \multirow{6}{*}{\(
                                    \begin{psmallmatrix}
                                      1 & 0 & 0 & 0 & 0 & 0 \\
                                      0 & s_{0} & 0 & 0 & 0 & 0 \\
                                      0 & 0 & s_{1} & 0 & 0 & 0 \\
                                      0 & 0 & s_{2} & s_{3} & 0 & 0 \\
                                      0 & 0 & s_{4} & s_{5} & s_{6} & 0 \\
                                      0 & 0 & 0 & s_{7} & s_{8} & s_{9}
                                    \end{psmallmatrix}
                                    \)}  & 0 & 1  &   & - & - \\ \cline{3-7}
            & & 1 & 5  &   & - & - \\ \cline{3-7}
            & & 2 & 10 &   & - & - \\ \cline{3-7}
            & & 3 & 9  & 1 & - & - \\ \cline{3-7}
            & & 4 & 3  & 2 & - & - \\ \cline{3-7}
            & & 5 &    & 1 & - & - \\ 
    \hline 
    \multirow{6}{*}{$\mf h_{25}$} & \multirow{6}{*}{\(
                                    \begin{psmallmatrix}
                                      1 & 0 & 0 & 0 & 0 & 0 \\
                                      0 & 1 & 0 & 0 & 0 & 0 \\
                                      0 & 0 & s_{0} & 0 & 0 & 0 \\
                                      0 & 0 & s_{1} & s_{2} & 0 & 0 \\
                                      0 & 0 & s_{3} & s_{4} & s_{5} & 0 \\
                                      0 & 0 & 0 & s_{6} & s_{7} & s_{8}
                                    \end{psmallmatrix}
                                    \)} & 0 & 1 & & & - \\ \cline{3-7}
            & & 1 & 5 &   &   & - \\ \cline{3-7}
            & & 2 & 9 & 1 &   & - \\ \cline{3-7}
            & & 3 & 5 & 5 &   & - \\ \cline{3-7}
            & & 4 &   & 4 & 1 & - \\ \cline{3-7}
            & & 5 &   &   & 1 & - \\
    \hline
    \multirow{8}{*}{$\mf h_{27}$} & \multirow{8}{*}{\(
                                    \begin{psmallmatrix}
                                      1 & 0 & 0 & 0 & 0 & 0 \\
                                      0 & 1 & 0 & 0 & 0 & 0 \\
                                      s_{0} & s_{1} & s_{2} & 0 & 0 & 0 \\
                                      0 & 0 & s_{3} & s_{4} & 0 & 0 \\
                                      0 & 0 & s_{5} & s_{6} & s_{7} & 0 \\
                                      0 & 0 & 0 & s_{8} & s_{9} & s_{10}
                                    \end{psmallmatrix}
                                    \)}  & 0 & 1 & & & - \\ \cline{3-7}
            & & 1 & 7  &    &   & - \\ \cline{3-7}
            & & 2 & 21 &    &   & - \\ \cline{3-7}
            & & 3 & 34 & 1  &   & - \\ \cline{3-7}
            & & 4 & 30 & 5  &   & - \\ \cline{3-7}
            & & 5 & 11 & 10 &   & - \\ \cline{3-7}
            & & 6 &    & 6  & 1 & - \\ \cline{3-7}
            & & 7 &    &    & 1 & - \\ \hline
  \end{tabular}
  \label{tab:sigmas-4step-no-h26}
\end{table}

\begin{table}[ht]
  \centering
  \caption{Moduli space of left-invariant metrics and isometric
    automorphisms for $5$-step CSLATs }
  \begin{tabular}[t]{|c|c|c|c|c|c|}
    \hline
    \parbox[m][1.2pc][c]{1.1pc}{$\mf h$} & $\Sigma$ & $\#\operatorname{nd}(\Sigma)=0$ & $\{e\}$ & $\mathbb{Z}_2$ & $\mathbb{Z}_2^2$ \\
    \hline\hline
    \multirow{7}{*}{$\mathfrak{h}_{28}$} & \multirow{7}{*}{\(
                                           \begin{psmallmatrix}
                                             1 & 0 & 0 & 0 & 0 & 0 \\
                                             0 & 1 & 0 & 0 & 0 & 0 \\
                                             0 & 0 & s_{0} & 0 & 0 & 0 \\
                                             0 & 0 & s_{1} & s_{2} & 0 & 0 \\
                                             0 & 0 & s_{3} & s_{4} & s_{5} & 0 \\
                                             0 & 0 & s_{6} & s_{7} & s_{8} & s_{9}
                                           \end{psmallmatrix}
                                           \)} & 0 &  & 1 & \\ \cline{3-6}
            & & 1 & & 6  & \\ \cline{3-6} 
            & & 2 & & 15 & \\ \cline{3-6}
            & & 3 & & 20 & \\ \cline{3-6}
            & & 4 & & 14 & 1 \\ \cline{3-6}
            & & 5 & & 4  & 2 \\ \cline{3-6}
            & & 6 & &    & 1 \\
    \hline
    \multirow{7}{*}{$\mathfrak{h}_{29}$} & \multirow{7}{*}{\(
                                           \begin{psmallmatrix}
                                             1 & 0 & 0 & 0 & 0 & 0 \\
                                             0 & s_{0} & 0 & 0 & 0 & 0 \\
                                             0 & 0 & s_{1} & 0 & 0 & 0 \\
                                             0 & 0 & s_{2} & s_{3} & 0 & 0 \\
                                             0 & 0 & s_{4} & s_{5} & s_{6} & 0 \\
                                             0 & 0 & s_{7} & s_{8} & s_{9} & s_{10}
                                           \end{psmallmatrix} 
                                           \)} & 0 & 1 &  & - \\ \cline{3-6}
            & & 1 & 6  &   & - \\ \cline{3-6}
            & & 2 & 15 &   & - \\ \cline{3-6}
            & & 3 & 20 &   & - \\ \cline{3-6}
            & & 4 & 14 & 1 & - \\ \cline{3-6}
            & & 5 & 4  & 2 & - \\ \cline{3-6}
            & & 6 &    & 1 & - \\
    \hline
    \multirow{8}{*}{$\mathfrak{h}_{30}$} & \multirow{8}{*}{\(
                                           \begin{psmallmatrix}
                                             1 & 0 & 0 & 0 & 0 & 0 \\
                                             s_{0} & s_{1} & 0 & 0 & 0 & 0 \\
                                             0 & 0 & s_{2} & 0 & 0 & 0 \\
                                             0 & 0 & s_{3} & s_{4} & 0 & 0 \\
                                             0 & 0 & s_{5} & s_{6} & s_{7} & 0 \\
                                             0 & 0 & s_{8} & s_{9} & s_{10} & s_{11}
                                           \end{psmallmatrix}
                                           \)} & 0 & 1 & & - \\ \cline{3-6}
            & & 1 & 7  &   & - \\ \cline{3-6}
            & & 2 & 21 &   & - \\ \cline{3-6}
            & & 3 & 35 &   & - \\ \cline{3-6}
            & & 4 & 35 &   & - \\ \cline{3-6}
            & & 5 & 20 & 1 & - \\ \cline{3-6}
            & & 6 & 5  & 2 & - \\ \cline{3-6}
            & & 7 &    & 1 & - \\ 
    \hline
    \multirow{9}{*}{$\mathfrak{h}_{31}$} & \multirow{9}{*}{\(
                                           \begin{psmallmatrix}
                                             1 & 0 & 0 & 0 & 0 & 0 \\
                                             s_{0} & 1 & 0 & 0 & 0 & 0 \\
                                             0 & 0 & s_{1} & 0 & 0 & 0 \\
                                             0 & s_{2} & s_{3} & s_{4} & 0 & 0 \\
                                             0 & 0 & s_{5} & s_{6} & s_{7} & 0 \\
                                             0 & 0 & s_{8} & s_{9} & s_{10} & s_{11}
                                           \end{psmallmatrix}
                                           \)} & 0 & 1 & & \\ \cline{3-6}
            & & 1 & 8  &   &   \\ \cline{3-6}
            & & 2 & 28 &   & \\ \cline{3-6}
            & & 3 & 56 &   & \\ \cline{3-6}
            & & 4 & 67 &  3 & \\ \cline{3-6}
            & & 5 & 44 & 12 & \\ \cline{3-6}
            & & 6 & 12 & 15 & 1 \\ \cline{3-6}
            & & 7 &    & 6 & 2  \\ \cline{3-6}
            & & 8 &    &   & 1  \\
    \hline
    \multirow{9}{*}{$\mathfrak{h}_{32}$} & \multirow{9}{*}{\(
                                           \begin{psmallmatrix}
                                             1 & 0 & 0 & 0 & 0 & 0 \\
                                             s_{0} & s_{1} & 0 & 0 & 0 & 0 \\
                                             0 & 0 & s_{2} & 0 & 0 & 0 \\
                                             0 & s_{3} & s_{4} & s_{5} & 0 & 0 \\
                                             0 & 0 & s_{6} & s_{7} & s_{8} & 0 \\
                                             0 & 0 & s_{9} & s_{10} & s_{11} & s_{12}
                                           \end{psmallmatrix}
                                           \)} & 0 & 1 & & - \\ \cline{3-6}
            & & 1  & 8  &   & - \\ \cline{3-6}
            & & 2  & 28 &   & - \\ \cline{3-6}
            & & 3  & 56 &   & - \\ \cline{3-6}
            & & 4  & 69 & 1 & - \\ \cline{3-6}
            & & 5  & 52 & 4 & - \\ \cline{3-6}
            & & 6  & 22 & 6 & - \\ \cline{3-6}
            & & 7  & 4  & 4 & - \\ \cline{3-6}
            & & 8  &    & 1 & - \\ \hline
  \end{tabular}
  \label{tab:sigmas-5step}
\end{table}

\begin{table}[ht]
  \centering
  \caption{Moduli space of left-invariant metrics for $\mf h_{13}$,
    $\mf h_{19}^+$ and $\mf h_{26}^-$}
  \begin{tabular}{|c|c|}
    \hline
    $\mathfrak h$ & $\Sigma$ \\
    \hline\hline
    \parbox[m][4.2pc][c]{1.1pc}{$\mf h_{13}$} & $\begin{psmallmatrix}
      1 & 0 & 0 & 0 & 0 & 0 \\
      s_{0} & 1 & 0 & 0 & 0 & 0 \\
      0 & 0 & s_{1} & 0 & 0 & 0 \\
      0 & 0 & s_{2} & s_{3} & 0 & 0 \\
      0 & 0 & 0 & s_{4} & s_{5} & 0 \\
      0 & 0 & 0 & s_{6} & s_{7} & s_{8}
    \end{psmallmatrix}$ \\
    \hline
    \parbox[m][4.2pc][c]{1.1pc}{$\mathfrak{h}_{19}^{+}$} & \(
                              \begin{psmallmatrix}
                                1 & 0 & 0 & 0 & 0 & 0 \\
                                s_{0} & 1 & 0 & 0 & 0 & 0 \\
                                s_{1} & s_{2} & 1 & 0 & 0 & 0 \\
                                0 & 0 & 0 & s_{3} & 0 & 0 \\
                                s_{4} & 0 & 0 & s_{5} & s_{6} & 0 \\
                                0 & 0 & 0 & s_{7} & s_{8} & s_{9}
                              \end{psmallmatrix}
                              \) \\
    \hline
    \parbox[m][4.2pc][c]{1.1pc}{$\mathfrak{h}_{26}^{-}$} & \(
                              \begin{psmallmatrix}
                                1 & 0 & 0 & 0 & 0 & 0 \\
                                s_{0} & 1 & 0 & 0 & 0 & 0 \\
                                0 & 0 & s_{1} & 0 & 0 & 0 \\
                                0 & 0 & s_{2} & s_{3} & 0 & 0 \\
                                0 & 0 & s_{4} & s_{5} & s_{6} & 0 \\
                                0 & 0 & s_{7} & s_{8} & s_{9} & s_{10}
                              \end{psmallmatrix}
                              \) \\ \hline
  \end{tabular}
  \label{tab:sigmas4}
\end{table}

\bibliographystyle{amsalpha}
\bibliography{csla.bib}

\end{document}